\documentclass[12pt]{amsart}
\usepackage[utf8]{inputenc}

\usepackage[margin=1.3in]{geometry}

\usepackage{amsmath}
\usepackage{amsfonts}
\usepackage{amssymb}
\usepackage{amsthm}
\usepackage{mathtools}
\usepackage{caption}
\usepackage{subcaption}
\usepackage{bbm}
\usepackage[export]{adjustbox}

\usepackage{stmaryrd}

\usepackage[all]{xy}

\usepackage{tikz-cd}
\usetikzlibrary{matrix}
\usepackage{graphicx} 
\usepackage{epstopdf}

\usepackage[linktocpage]{hyperref}
\hypersetup{
    colorlinks=true,
    linkcolor=blue,
    citecolor=blue,      
    urlcolor=blue,
}

\usepackage{color}
\definecolor{note}{rgb}{0,0,1}  

\newtheorem{theorem}{Theorem}
\newtheorem{definition}[theorem]{Definition}
\newtheorem{proposition}[theorem]{Proposition}
\newtheorem{lemma}[theorem]{Lemma}
\newtheorem{claim}[theorem]{Claim}
\newtheorem{corollary}[theorem]{Corollary}
\newtheorem{example}[theorem]{Example}
\newtheorem{remark}[theorem]{Remark}
\newtheorem{conjecture}[theorem]{Conjecture}

\numberwithin{equation}{section}
\numberwithin{theorem}{section}

\usepackage{enumitem}
\newcommand{\mylabel}[2]{#2\def\@currentlabel{#2}\label{#1}}

\usepackage{todonotes}

\usepackage[english]{babel}

\usepackage{hyphenat}

\usepackage[backend=bibtex,style=alphabetic,maxalphanames=4,maxnames=4]{biblatex}

\renewbibmacro{in:}{}

\DeclareDelimFormat[bib,biblist]{nametitledelim}{\addcomma\space}

\DeclareFieldFormat*{title}{\mkbibitalic{#1}\addcomma}
\DeclareFieldFormat*{journaltitle}{#1}
\DeclareFieldFormat*{volume}{\mkbibbold{#1}}
\DeclareFieldFormat{pages}{#1}
\DeclareFieldFormat[misc]{date}{preprint {#1}}
\DeclareFieldFormat{mr}{%
  MR\addcolon\space
  \ifhyperref
    {\href{http://www.ams.org/mathscinet-getitem?mr=MR#1}{\nolinkurl{#1}}}
    {\nolinkurl{#1}}}
    
\AtEveryBibitem{
  \clearfield{url}
  \clearfield{number}
  \clearfield{doi}
  \clearfield{issn}
  \clearfield{isbn}
  \clearfield{eprintclass}
}
\AtEveryBibitem{\ifentrytype{book}{\clearfield{pages}}{}}

\bibliography{hecke}

\usepackage{fancyhdr}
\usepackage{todonotes}

\title{A link invariant from higher-dimensional Heegaard Floer homology}
\author{Tianyu Yuan}
\address{Beijing International Center for Mathematical Research, Peking University, Beijing 100871, China}
\email{ytymath@pku.edu.cn} \urladdr{}
\date{\today}

\keywords{Higher-dimensional Heegaard Floer homology, Khovanov homology}

\subjclass[2010]{Primary 53D40; Secondary 57M27.}

\begin{document}

\maketitle

\begin{abstract}

We define a higher-dimensional analogue of symplectic Khovanov homology. Consider the standard Lefschetz fibration $p\colon W\to D\subset\mathbb{C}$ of a $2n$-dimensional Milnor fiber of the $A_{2\kappa-1}$ singularity. We represent a link by a $\kappa$-strand braid, which is expressed as an element $h$ of the symplectic mapping class group $\mathrm{Symp}(W,\partial W)$. We then apply the higher-dimensional Heegaard Floer homology machinery to the pair $(\boldsymbol{a},h(\boldsymbol{a}))$, where $\boldsymbol{a}$ is a collection of $\kappa$ unstable manifolds of $W$ which are Lagrangian spheres. We prove its invariance under arc slides and Markov stabilizations, which shows that it is a link invariant.
This work constitutes part of the author's PhD thesis.

\end{abstract}

\tableofcontents

\section{Introduction}

Many powerful Floer-theoretic invariants of knots and links have emerged over the past two decades. 
These include Heegaard Floer homology \cite{ozsvath2004holomorphic} and knot Floer homology in dimension 1; symplectic Khovanov homology \cite{seidel2006link} and knot contact homology \cite{ekholm2013knot} in dimension 2. 
Here when we say ``dimension $n$'', we are taking the ambient symplectic manifold to be $2n$-dimensional and the Lagrangian submanifolds (if we are talking about Lagrangian intersection Floer thoeries) to be $n$-dimensional. In \cite{manolescu2007link} Manolescu also used quiver varieties to define a higher-dimensional analogue of $\mathfrak{sl}(n)$-homologies. 

Along similar lines, the aim of this paper is to construct a link invariant using higher-dimensional Heegaard Floer homology, which is defined in \cite{colin2020applications} as a higher-dimensional analogue of Heegaard Floer homology in a cylindrical setting. 

More specifically, in dimension 1, Lipschitz \cite{lipshitz2006cylindrical} proved the equivalence between Ozsv\'ath and Szab\'o's Heegaard Floer homology \cite{ozsvath2004holomorphic} and its cylindrical analogue. In dimension 2, Mak and Smith \cite{mak2020fukayaseidel} established the equivalence of symplectic Khovanov homology and its cylindrical interpretation. Colin, Honda, and Tian \cite{colin2020applications} then defined a higher-dimensional analogue of cylindrical Heegaard Floer homology, which helps place the cylindrical symplectic Khovanov homology in a more general framework. 
In this paper, the ambient manifold is a $2n$-dimensional Milnor fiber of the $A_{2\kappa-1}$-singularity $p:W\to D\subset\mathbb{C}$, extending the case of $n=1$ considered in \cite{colin2020applications}. 
Given a link, we consider its $\kappa$-strand braid representation $\sigma$, which corresponds to an element $h$ of the symplectic mapping class group $\mathrm{Symp}(W,\partial W)$. 
There is a natural collection of $\kappa$ Lagrangian spheres $\boldsymbol{a}$ by the matching cycle construction between pairs of critical points of $p$. 
We then apply the higher-dimensional Heegaard Floer homology machinery to the pair $(\boldsymbol{a},h(\boldsymbol{a}))$ to define the link invariant and denote the homology group by $Kh^\sharp(\widehat{\sigma})$.

Though cylindrical versions of Heegaard Floer theories are more convenient for visualizing pseudoholomorphic curves, the original theories defined in the symmetric products $\mathrm{Sym}^\kappa(M)$ have their advantages: In dimension 1, Perutz \cite{perutz2008hamiltonian} proved that Lagrangians in $\mathrm{Sym}^\kappa(\Sigma)$ related by a handle slide are in fact Hamiltonian isotopic for some specific symplectic form, which directly implies the handle slide invariance property without curve counting techniques in \cite{ozsvath2004holomorphic}. 
In dimension 2, Seidel and Smith \cite{seidel2006link} considered nilpotent slices instead of $\mathrm{Sym}^\kappa(M)$, which was shown to be a subset of $\mathrm{Hilb}^\kappa(M)$ by Manolescu \cite{manolescu2006nilpotent}. 
Inside the nilpotent slice, matching cycles as Lagrangians related by arc slides are also Hamiltonian isotopic, which is not obvious in the cylindrical formulation. 
Mak and Smith \cite{mak2020fukayaseidel} then showed that the cylindrical version is equivalent to the original symplectic Khovanov homology in nilpotent slices. 

However, in higher dimensions, the cylindrical symplectic Khovanov homology with vanishing cycles $\kappa$-tuples of $S^{n}$ does not have its ``original'' version; at this moment we do not know how to put the theory inside a nilpotent slice setting. 
Another problem is that Hilbert schemes of points in higher dimensions are not smooth, so we need new ways to resolve the singularities along the diagonal of $\mathrm{Sym}^\kappa(M)$.
One possible solution is to restrict to some smooth stratum of $\mathrm{Hilb}^\kappa(M)$. 
For example, the subset of subschemes where each support point is of length at most 3 (at most triple point) is smooth. 
However we will not continue the discussion in this paper further. 
It would be interesting to study this problem in future.
~\\

\noindent
Therefore, we will adopt the curve counting approach as in \cite{ozsvath2004holomorphic} and \cite{colin2020applications} to prove the invariance under arc slides and Markov stabilizations. 

In Section \ref{section-definition}, we begin with a brief review of Section 9 of \cite{colin2020applications}, providing necessary notations, definitions and prerequisite theorems. 
Then we state the main result. We use a subsection to explain the Morse flow tree theory and its relation to pseudoholomorphic curves originated from \cite{fukaya1997zero}, which is crucial in our proof. 

In Section \ref{section-arc}, we show the arc slide invariance by counting pseudoholomorphic curves with certain boundary Lagrangians. The idea is to stretch the curve into several parts so that each one is easy to count by elementary model calculation.

In Section \ref{section-model}, we perform a model calculation of pseudoholomorphic quadrilaterals, which will be used in Section \ref{section-markov} and Section \ref{section-example} for several times.

In Section \ref{section-markov}, we translate the Markov stabilization into the gluing of pseudoholomorphic curves, which we count by Morse flow tree arguments instead. 

In Section \ref{section-example}, we compute $Kh^\sharp(\widehat{\sigma})$ for unknots, Hopf links and trefoils.
~\\

\noindent
\textbf{Acknowledgements.} I would like to thank Ko Honda for countless discussions and introducing this project to me. I also thank Yin Tian for numerous ideas and suggestions, and thank Eilon Reisin-Tzur for his patient revision. 
I am partially supported by China Postdoctoral Science Foundation 2023T160002.

\section{Definitions and main results}
\label{section-definition}
\subsection{Higher-dimensional analogue of symplectic Khovanov homology}
\label{subsection-definition}
This subsection works as a review of \cite{colin2020applications}, so most proofs and details are omitted. 

Let $\widetilde{D}=\{-2\leq\mathrm{Re}\,z,\mathrm{Im}\,z\leq 2\}\subset\mathbb{C}_z$. 
We consider the standard $2n$-dimensional Lefschetz fibration
\begin{equation*}
   \widetilde{p}\colon\widetilde{W}\to\widetilde{D}\subset\mathbb{C}_z
\end{equation*}
for a Milnor fiber of the $A_{2k-1}$ singularity, where the regular fiber is $T^*S^{n-1}$. 
There are $2\kappa$ critical values $\boldsymbol{\widetilde{z}}=\{z_1,\dots,z_{2\kappa}\}$, where $\mathrm{Re}\,z_i=\mathrm{Re}\,z_{i+\kappa}$, $\mathrm{Im}\,z_i=-1$ and $\mathrm{Im}\,z_{i+\kappa}=1$, $i=1,\dots,\kappa$. Let
\begin{equation*}
   p\colon W\coloneqq\widetilde{p}^{-1}(D)\to{D}
\end{equation*}
be the restriction of $\widetilde{p}$ to $D=\widetilde{D}\cap \{\mathrm{Im}\,z\leq 0\}$. 
For $i=1,\dots,\kappa$, connect $z_i$ and $z_{i+\kappa}$ by straight arcs $\widetilde{\gamma}_i$. 
Then the matching cycles $\widetilde{a}=\{\widetilde{a}_1,\dots,\widetilde{a}_\kappa\}$ over $\{\widetilde{\gamma}_1,\dots,\widetilde{\gamma}_\kappa\}$ are Lagrangian spheres. 
Let $\boldsymbol{z}$, $a_i$, and $\gamma_i$ be ``half'' of $\boldsymbol{\widetilde{z}}$, $\widetilde{a}_i$, and $\widetilde{\gamma}_i$, i.e., their restrictions to $W$. 

Given a $\kappa$-strand braid $\sigma\in\mathrm{Diff}^+(D,\partial D,\boldsymbol{z})$, let $h_\sigma\in\mathrm{Symp}(W,\partial W)$ be the monodromy on $W$ which descends to $\sigma$ and let $\widetilde{h}_\sigma$ be the extension of $h_\sigma$ to $\widetilde{W}$ by identity.

In this paper we always do cohomology. The variant $CKh^\sharp(\widehat{\sigma})$ of the symplectic Khovanov cochain complex is defined as the higher-dimensional Heegaard Floer cochain complex, in the sense of \cite{lipshitz2006cylindrical} and \cite{colin2020applications}, denoted by $\widehat{CF}(\widetilde{W},\widetilde{h}_\sigma(\widetilde{\boldsymbol{a}}),\widetilde{\boldsymbol{a}})$. 
Specifically, a $\kappa$-tuple of intersection points of $\widetilde{h}_\sigma(\widetilde{\boldsymbol{a}})$ and $\widetilde{\boldsymbol{a}}$ is a $\kappa$-tuple $\boldsymbol{y}=\{y_1,\dots,y_\kappa\}$ where $y_i\in\widetilde{a}_i\cap\widetilde{h}_\sigma(\widetilde{a}_{\beta(i)})$ and $\beta$ is some permutation of $\{1,\dots,\kappa\}$. 
Then $\widehat{CF}(\widetilde{W},\widetilde{h}_\sigma(\widetilde{\boldsymbol{a}}),\widetilde{\boldsymbol{a}})$ is the free $\mathbb{F}[\mathcal{A}]\llbracket\hbar,\hbar^{-1}]$-module generated by all such $\kappa$-tuples $\boldsymbol{y}$, where the coefficient ring is discussed below.

To define the differential, let $F$ be a surface with boundary of Euler characteristic $\chi$ and $\dot{F}$ be the surface with boundary punctures. 
We start with the split almost complex structure $J_{\mathbb{R}\times[0,1]}\times J_{\widetilde{W}}$ on $\mathbb{R}\times[0,1]\times\widetilde{W}$, and apply a perturbation to achieve transversality, denoted by $J^{\lozenge}$. 

For $\boldsymbol{y},\boldsymbol{y}'\in\widehat{CF}(\widetilde{W},\widetilde{h}_\sigma(\widetilde{\boldsymbol{a}}),\widetilde{\boldsymbol{a}})$, let $\mathcal{M}^{\mathrm{ind}=1,A,\chi}_{J^\lozenge}(\boldsymbol{y},\boldsymbol{y}')$ be the moduli space of $u\colon\dot F\to\mathbb{R}\times[0,1]\times\widetilde{W}$ satisfying

\begin{enumerate}
    \item $du\circ J^\lozenge=J^\lozenge\circ du$;
    \item $u(\partial\dot F)\subset\mathbb{R}\times((\{1\}\times\widetilde{\boldsymbol{a}})\cup(\{0\}\times\widetilde{h}_\sigma(\widetilde{\boldsymbol{a}})))$;
    \item As $\pi_\mathbb{R}\circ u$ tends to $+\infty$ (resp. $-\infty$), $\pi_{\widetilde{W}}$ tends to $\boldsymbol{y}$ (resp. $\boldsymbol{y}'$), where $\pi_{\mathbb{R}}, \pi_{\widetilde{W}}$ are the projections of $u$ to $\mathbb{R}$ and $\widetilde{W}$.
\end{enumerate}

The differential is then defined as
\begin{equation}
\label{diff}
    d\boldsymbol{y}=\sum_{\boldsymbol{y}',\chi\leq\kappa,A\in\mathcal{A}}\#\mathcal{M}^{\mathrm{ind}=1,A,\chi}_{J^\lozenge}(\boldsymbol{y},\boldsymbol{y}')/\mathbb{R}\cdot\hbar^{\kappa-\chi}\cdot e^A\cdot\boldsymbol{y}'.
\end{equation}

We write $Kh^\sharp(\widehat{\sigma})$ for the cohomology group $\widehat{HF}(\widetilde{W},\widetilde{h}_\sigma(\widetilde{\boldsymbol{a}}),\widetilde{\boldsymbol{a}})$.

The coefficient ring $\mathbb{F}[\mathcal{A}]\llbracket\hbar,\hbar^{-1}]$ (power series in $\hbar$ and polynomial in $\hbar^{-1}$) keeps track of the relative homology class and Euler characteristic of the domain, where
\begin{equation*}
    \mathcal{A}=H_2([0,1]\times\widetilde{W},(\{1\}\times\widetilde{\boldsymbol{a}})\cup(\{0\}\times\widetilde{h}_\sigma(\widetilde{\boldsymbol{a}}));\mathbb{Z}).
\end{equation*}

The following lemmas justify the use of coefficient $\mathbb{F}[\mathcal{A}]\llbracket\hbar,\hbar^{-1}]$ for $n=2$ and $\mathbb{F}\llbracket\hbar,\hbar^{-1}]$ for $n>3$:

\begin{lemma}
    For fixed ${\boldsymbol{y}}$,${\boldsymbol{y}'}$ and $\chi$, $\#\mathcal{M}^{\operatorname{ind}=1,\chi}_{J^\lozenge}(\textbf{y},\textbf{y}')/\mathbb{R}$ is finite.
\end{lemma}

\begin{lemma}
\label{lemma-coeff}
    Suppose $\widetilde{W}$ is of dimension $2n$. For $n=2$, $\mathcal{A}\simeq\mathbb{Z}^{r-1}$, where $r$ is the number of connected components of $\widehat{\sigma}$; for $n>2$, $\mathcal{A}\simeq\{0\}$.
\end{lemma}

\begin{proof}
    Denote $X=[0,1]\times\widetilde{W}$, $Y=(\{1\}\times\widetilde{\boldsymbol{a}})\cup(\{0\}\times\widetilde{h}_\sigma(\widetilde{\boldsymbol{a}})))$, and $i\colon Y\to X$ for the inclusion.

    If $n=2$, $H_3(X,Y)$ and $H_1(Y)$ are trivial. By the long exact sequence for relative singular homology, $H_2(X,Y)\simeq H_2(X)/i_*H_2(Y)$. 
    Since there are $2\kappa$ critical points of $\widetilde{p}\colon\widetilde{W}\to\widetilde{D}$, $H_2(X)$ is generated by $2\kappa-1$ matching cycles over arcs connecting pairs of critical values in $\widetilde{D}$. 
    $H_2(Y)$ is simply generated by collections $\widetilde{\boldsymbol{a}}$ and $\widetilde{h}_\sigma(\widetilde{\boldsymbol{a}})$. 
    Therefore, $H_2(X,Y)$ is generated by $2\kappa-1$ arcs quotient by collapsing arcs corresponding to $\widetilde{\boldsymbol{a}}$ and $\widetilde{h}_\sigma(\widetilde{\boldsymbol{a}})$. 
    The number of remaining nontrivial arcs is $r-1$, where $r$ is the number of connected components of the braid $\widehat{\sigma}$. 
    
    If $n>2$, $H_2(X)$ and $H_1(Y)$ are trivial, so $H_2(X,Y)\simeq\{0\}$ by the relative homology exact sequence. 
\end{proof}

The following lemma computes the Fredholm index, which is a generalization from \cite{colin2012equivalence} and the proof is omitted:
\begin{lemma}
    The Fredholm index of $u$ is
    \begin{equation*}
        \mathrm{ind}(u)=(n-2)(\chi-\kappa)+\mu(u),
    \end{equation*}
    where $\mu(u)$ is the Maslov index.
    \label{lemma-index}
\end{lemma}

The Floer homology group $Kh^\sharp(\widehat{\sigma})$ is well-defined and now we state our main result, which is a higher-dimensional version of Theorem 1.2.1 of \cite{colin2020applications}:
\begin{theorem}
    \label{main-thm}
    For $n=2$ or $n>3$, $Kh^\sharp(\widehat{\sigma})$ is a link invariant, i.e., it is independent of the choice of arcs $\{\widetilde{\gamma}_1,\dots,\widetilde{\gamma}_\kappa\}$ and Lagrangian thimbles $\{\widetilde{a}_1,\dots,\widetilde{a}_\kappa\}$, and is invariant under Markov stabilizations.
\end{theorem}

\begin{remark}
    $Kh^\sharp(\widehat{\sigma})$ is a relative graded module over $\mathbb{F}[\mathcal{A}]\llbracket\hbar,\hbar^{-1}]$. From (\ref{diff}) and Lemma \ref{lemma-index}, $\hbar$ is of degree $2-n$. 
\end{remark}

\begin{remark}
    Note that in Theorem \ref{main-thm}, the case of $n=3$ is excluded. The reason is explained by the remark after Lemma \ref{lemma-cycle} in Section \ref{section-arc}.
\end{remark}

The computation of $Kh^\sharp(\widehat{\sigma})$ for simple links in Section \ref{section-example} gives results highly similar to Khovanov homology \cite{khovanov2000categorification, williams2008computations}. We will not discuss their relation further in this paper. Instead, we leave it as a conjecture:

\begin{conjecture}
    \begin{equation*}
        Kh^{\sharp,k}(\widehat{\sigma})\simeq\bigoplus_{i-j=k\,\operatorname{mod}\,n-2}Kh^{i,j}(\sigma).
    \end{equation*}
    \label{conj}
\end{conjecture}

Note that Conjecture \ref{conj} implies that we have found nothing new other than the symplectic Khovanov homology \cite{seidel2006link}. While it may serve as a new approach to compute the symplectic Khovanov homology.

\subsection{Moduli space of Morse gradient trees}
\label{subsection-gradient}
If two closed Lagrangian submanifolds $L_0,L_1$ of a symplectic manifold $(M,\omega)$ are $C^1$-close to each other, and moreover they are exact Lagrangian isotopic, then by the Lagrangian neighbourhood theorem, one can view $L_1$ as the graph $\Gamma_{df}\subset T^*L_0$, where $f\colon L_0\to\mathbb{R}$. 
In this case, the count of pseudoholomorphic disks bounding $L_0,L_1$ can be reduced to the count of Morse gradient trajectories on $L_0$, which is much more convenient to compute. 
More generally, if we have exact Lagrangian isotopic $L_0,\dots,L_{k-1}$, where $L_i=\Gamma_{df_i}\subset T^*L_0$ for $f_i\colon L_0\to\mathbb{R},\,i=0,\dots,k-1$, then we count gradient trees, to be explained below. 
This alternative method of counting appears several times in Section \ref{section-arc} and \ref{section-markov}, so we make the statement precise here. 
The main reference is \cite{fukaya1997zero}, which generalizes Floer's correspondence between gradient trajectories and pseudoholomorphic strips \cite{floer1988morse}. 
There is also a simpler proof of \cite{fukaya1997zero} by Iacovino \cite{iacovino2008simple}, which considers the perturbation of Floer equations.

\begin{figure}[ht]
    \centering
    \includegraphics[height=6cm]{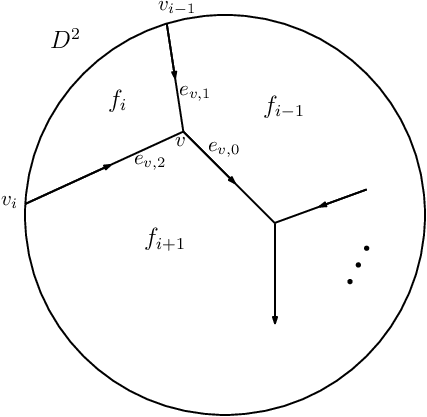}
    \caption{An embedding of $T$ into $D^2$.}
    \label{tree-definition}
\end{figure}

\begin{definition}
    A ribbon tree is a directed tree $T$ which satisfies:
    \begin{enumerate}[label=(\arabic*)]
        \item $T$ has no degree 2 vertex.
        \item For each internal vertex $v\in T$, there is a cyclic order of the edges attached to $v$, i.e., we have a labeling bijection $\ell_v\colon E_v\to\{0,1,\dots,k_v-1\}$, where $E_v$ is the set of edges attached to $v$ and $k_v$ is the degree of $v$ such that there is a unique outgoing edge and it is labeled 0. 
        Therefore, there are two labels for each internal edge coming from the two labeling functions corresponding to the endpoints of the edge and one label for each external edge.
        \item There is a cyclic order of external vertices of $T$ such that the root has label 0, i.e., we have a bijection $\ell\colon V_{ext}\to\{0,\dots,k-1\}$ where $V_{ext}$ is the set of external vertices and the one adjacent to the unique outgoing edge is labeled by 0.
        \item There exists an embedding $\phi\colon T\to D^2$ so that at each internal vertex $v$, $\phi(\ell_v^{-1}(0)),\dots,\phi(\ell_v^{-1}(k_v-1))$ are in counterclockwise order around $\phi(v)$. 
        Also, $\phi(V_{ext})\subset\partial D^2$ and $\phi(\ell^{-1}(0)),\dots,\phi(\ell^{-1}(k-1))$ are in counterclockwise order on $\partial D^2$. Note that the way we embed $T$ is not important. 
    \end{enumerate}
    In the definition univalent vertices are called external and the others are internal; edges adjacent to external vertices are called external and the others are internal.
\end{definition}

The edges of $T$ will represent gradient flows. 
Specifically, let $(M,g)$ be a Riemannian manifold and $f_0,\dots,f_{k-1}$ be $C^\infty$-functions on $M$ so that $f_{i}-f_{i+1}$ is Morse for each $i$ (where $f_{k}=f_0$). We define the moduli space of gradient trees $\mathcal{M}_g(M;\boldsymbol{f},\boldsymbol{p})$ to consist of pairs $(T,I)$ where $T$ is a ribbon tree and $I\colon T\backslash V_{ext}\to M$ is a continuous function satisfying:
\begin{enumerate}[label=(\arabic*)]
    \item $\lim_{x\to v_i} I(x)=p_i$, where $p_i$ is a critical point of $f_{i}-f_{i+1}$.
    \item External edges are identified with $(-\infty,0]$ except the outgoing one which is identified with $[0,\infty)$; each internal edge $e$ is identified with $[0,t(e)]$ where $t:E_{int}\to[0,\infty)$ is a length function defined on internal edges. Now $I$ takes the above intervals to $M$. 
    For each edge $e$ of $T$,
    \begin{equation*}
        \dot{I}|_e=-\nabla_g(f_{l(e)}-f_{r(e)}),
    \end{equation*}
    where $l(e)$ and $r(e)$ are defined with respect to the direction of $e$, i.e. if one looks in the positive direction of $e$ on $D^2$, then $l(e)$ (resp. $r(e)$) is the component of $D^2\backslash T$ on the left (resp. right) side of $e$.
\end{enumerate}

Figure \ref{tree-definition} shows an element $T$, together with an embedding of $T$ in $D^2$. Note that we assign each function $f_i$ to a corresponding region of $D^2\backslash T$.

\begin{lemma}[\cite{fukaya1997zero}]
    Fixing a generic choice of $\boldsymbol{f}$, $\mathcal{M}_g(M,\boldsymbol{f},\boldsymbol{p})$ is a smooth manifold of dimension
    \begin{equation}
        \sum_{v}\mathrm{ind}(v)-(k-1)n+(k-3),
    \end{equation}
    where $k$ is the number of boundary vertices and $n$ is the dimension of the manifold.
    \label{lemma-morse}
\end{lemma}

Next we define the moduli space of pseudoholomorphic disks bounding exact Lagrangian submanifolds in $T^*M$. Let $(M,g)$ be a Riemannian manifold.
Let $\boldsymbol{f}=(f_0,\dots,f_{k-1})$ be a generic collection of functions on $M$, and let $\boldsymbol{\Gamma}=(\Gamma_{df_0},\dots,\Gamma_{df_{k-1}})$ be the graphs of their differentials.
We associate each critical point $p_i$ of $f_{i}-f_{i+1}$ with $x_i=(p_i,\epsilon df_i(p_i))\in\epsilon\Gamma_{df_{i+1}}\cap\epsilon\Gamma_{df_i}$ where $\epsilon>0$ is small.

We then fix a canonical almost complex structure $J_g$ on $T^*M$ associated to the metric $g$ on $M$ such that
\begin{enumerate}[label=(\arabic*)]
    \item $J_g$ is compatible with the canonical symplectic form $\omega$ on $T^*M$.
    \item $J_g$ maps vertical tangent vectors to horizontal tangent vectors of $T^*M$ with respect to $g$.
    \item On the zero section of $T^*M$, for $v\in T_q M\subset T_{(q,0)}(T^*M)$, let $J_g(v)=g(v,\cdot)\in T_q^*M\subset T_{(q,0)}(T^*M)$.
\end{enumerate}

\begin{definition}
Let $\mathcal{M}_{J_g}(T^*M;\epsilon\boldsymbol{\Gamma},\boldsymbol{x})$ consist of pairs $(u,D^2_{\boldsymbol{z}})$ where $D^2_{\boldsymbol{z}}$ is the domain $D^2$ with marked points $\boldsymbol{z}=(z_0,\dots,z_{k-1})$ arranged in counterclockwise order on $\partial D^2$ and $u\colon D^2\backslash\{z_0,\dots,z_{k-1}\}\to T^*M$ satisfies
\begin{enumerate}[label=(\arabic*)]
    \item $u(z_i)=p_i$,
    \item $u(\partial_i D^2)\subset\epsilon\Gamma_{df_i}$,
    \item $J_g\circ du=du\circ J_g$,
\end{enumerate}
where $\partial_i D^2$ is the shortest counterclockwise arc between $z_i$ and $z_{i+1}$ on $\partial D^2$. Here we are identifying pairs $(u,D^2_{\boldsymbol{z}})$ and $(v,D^2_{\boldsymbol{z}'})$ which are related by an isomorphism of the domain.
\end{definition}

The following theorem relates Morse gradient trees on $M$ to pseudoholomorphic disks on $T^*M$:
\begin{theorem}[\cite{fukaya1997zero}]
    For $\epsilon>0$ sufficiently small, there is an oriented diffeomorphism $\mathcal{M}_g(M;\boldsymbol{f},\boldsymbol{p})\cong\mathcal{M}_{J_g}(T^*M;\epsilon\boldsymbol{\Gamma},\boldsymbol{x})$.
    \label{Fukaya-Oh}
\end{theorem}

Roughly speaking, for each gradient tree $(T,I)$, we can construct a pseudoholomorphic curve near $I(T)$ inside $T^*M$ by some gluing techniques. 
We should note that the theorem depends on a specific choice of almost complex structure $J_g$. 

\begin{figure}[ht]
    \centering
    \includegraphics[width=13cm]{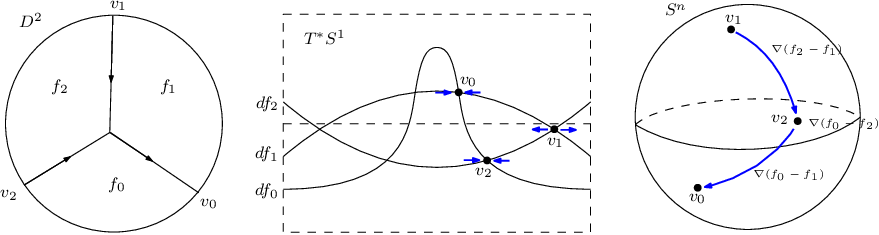}
    \caption{The gradient tree $T$ viewed inside $D^2$ (left), the perturbation of Lagrangians in the special case of $T^*S^1$ (middle) and the image of $T$ on $S^n$ (right). We abuse notation and label both the domain and image of a vertex by $v_i$ since there is no ambiguity.}
    \label{tree-example}
\end{figure}

The following example counts pseudoholomorphic triangles with 3 Lagrangian boundaries, which we will meet again in Section \ref{section-arc}.
\begin{example}
    \label{ex-tree}
    We give an example of a gradient tree with 3 vertices here. Consider functions $f_0,f_1,f_2$ on $S^n$ so that the domain $D^2$ looks like the left side of Figure \ref{tree-example}. We can perturb $f_0,f_1,f_2$ generically so that $v_1$ is a source (top generator) of $\nabla(f_2-f_1)$, $v_2$ is a sink (bottom generator) of $\nabla(f_0-f_2)$ and $v_0$ is a sink (bottom generator) of $\nabla(f_0-f_1)$. The middle of Figure \ref{tree-example} shows a possible perturbation when $n=1$ for illustration. Sources and sinks are denoted by blue arrows. The right side of Figure \ref{tree-example} shows what happens on $S^n$: $v_1$ is the unique top generator of $f_2-f_1$ on $S^n$, so the flows from $v_1$ form a $S^{n-1}$-family and pass through all the points of $S^n$ except two critical points. $v_2$ is the bottom generator of $f_0-f_2$, so the flow from $v_2$ of $\nabla(f_0-f_2)$ is of length 0. Therefore the flow from $v_1$ should pass $v_2$ and there is a unique such flow line. Similarly, there is a unique flow line from $v_2$ to $v_0$ of $\nabla(f_0-f_1)$. To conclude, there is a unique gradient tree. Thus there is a unique pseudoholomorphic disk bounded by the Lagrangians $\epsilon\Gamma_{df_0},\epsilon\Gamma_{df_1},\epsilon\Gamma_{df_2}$ for small $\epsilon$ by Theorem \ref{Fukaya-Oh}.
\end{example}

\section{Invariance under arc slides}
\label{section-arc}
Given $\{\widetilde{\gamma}_1,\dots,\widetilde{\gamma}_\kappa\}$ from Section \ref{subsection-definition}, consider the arc slide of $\widetilde{\gamma}_1$ over $\widetilde{\gamma}_2$: Let $\{\widetilde{\gamma}'_1,\dots,\widetilde{\gamma}'_\kappa\}$ be the new set of arcs as in Figure \ref{arc-slide}. Let $\widetilde{\boldsymbol{a}}'=\{\widetilde{a}'_1,\dots,\widetilde{a}'_\kappa\}$ be the new tuple of Lagrangians over $\{\widetilde{\gamma}'_1,\dots,\widetilde{\gamma}'_\kappa\}$. For $i=1,\dots,\kappa$, let $\Theta_i$ (resp. $\Xi_i$) be the intersection point of $\widetilde{a}_i$ and $\widetilde{a}'_i$ that lies over $z_i$ (resp. $z_{i+\kappa}$). Denote $\boldsymbol{\Theta}=\{\Theta_1,\dots,\Theta_\kappa\}$ and $\boldsymbol{\Xi}=\{\Xi_1,\dots,\Xi_\kappa\}$.

\begin{figure}[ht]
    \centering
    \includegraphics[height=4cm]{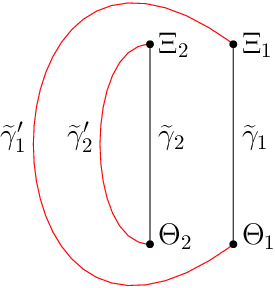}
    \caption{Arc sliding of $\widetilde{\gamma}_1$ over $\widetilde{\gamma}_2$.}
    \label{arc-slide}
\end{figure}

The purpose of this section is to prove:
\begin{theorem}
    $\widehat{CF}(\widetilde{W},\widetilde{h}_\sigma(\widetilde{\boldsymbol{a}}),\widetilde{\boldsymbol{a}})$ and $\widehat{CF}(\widetilde{W},\widetilde{h}_\sigma(\widetilde{\boldsymbol{a}}'),\widetilde{\boldsymbol{a}}')$ are quasi-isomorphic.
\end{theorem}
\begin{proof}
    It suffices to show the quasi-isomorphism between $\widehat{CF}(\widetilde{W},\widetilde{h}_\sigma(\widetilde{\boldsymbol{a}}),\widetilde{\boldsymbol{a}})$ and $\widehat{CF}(\widetilde{W},\widetilde{h}_\sigma(\widetilde{\boldsymbol{a}}),\widetilde{\boldsymbol{a}}')$.
    One can show the quasi-isomorphism between $\widehat{CF}(\widetilde{W},\widetilde{h}_\sigma(\widetilde{\boldsymbol{a}}),\widetilde{\boldsymbol{a}}')$ and $\widehat{CF}(\widetilde{W},\widetilde{h}_\sigma(\widetilde{\boldsymbol{a}}'),\widetilde{\boldsymbol{a}}')$ similarly.
    
    We define the cochain map by the $\mu_2$ composition map of the $A_\infty$-relation:
    \begin{equation*}
        \Phi\colon\widehat{CF}(\widetilde{W},\widetilde{h}_\sigma(\widetilde{\boldsymbol{a}}),\widetilde{\boldsymbol{a}})\to\widehat{CF}(\widetilde{W},\widetilde{h}_\sigma(\widetilde{\boldsymbol{a}}),\widetilde{\boldsymbol{a}}'),
    \end{equation*}
    \begin{equation*}
        \mathrm{\boldsymbol{y}}\mapsto\mu_2(\boldsymbol{\Xi}\otimes\mathrm{\boldsymbol{y}}),
    \end{equation*}
    \noindent
    where $\mu_2$ is the product map
    \begin{equation*}
        \mu_2\colon\widehat{CF}(\widetilde{W},\widetilde{\boldsymbol{a}},\widetilde{\boldsymbol{a}}')\otimes\widehat{CF}(\widetilde{W},\widetilde{h}_\sigma(\widetilde{\boldsymbol{a}}),\widetilde{\boldsymbol{a}})\to\widehat{CF}(\widetilde{W},\widetilde{h}_\sigma(\widetilde{\boldsymbol{a}}),\widetilde{\boldsymbol{a}}').
    \end{equation*}
    \noindent
    Similarly we define the cochain map going back by
    \begin{equation*}
        \Psi\colon\widehat{CF}(\widetilde{W},\widetilde{h}_\sigma(\widetilde{\boldsymbol{a}}),\widetilde{\boldsymbol{a}}')\to\widehat{CF}(\widetilde{W},\widetilde{h}_\sigma(\widetilde{\boldsymbol{a}}),\widetilde{\boldsymbol{a}}),
    \end{equation*}
    \begin{equation*}
        \mathrm{\boldsymbol{y}}\mapsto\mu'_2(\boldsymbol{\Theta}\otimes\mathrm{\boldsymbol{y}}).
    \end{equation*}
    
    The following lemma justifies that $\Phi$, $\Psi$ are well-defined:
    \begin{lemma}
        For $n=2$ or $n>3$, $\boldsymbol{\Xi}$ is a cocycle in $\widehat{CF}(\widetilde{W},\widetilde{\boldsymbol{a}},\widetilde{\boldsymbol{a}}')$ and $\boldsymbol{\Theta}$ is a cocycle in $\widehat{CF}(\widetilde{W},\widetilde{\boldsymbol{a}}',\widetilde{\boldsymbol{a}})$.
        \label{lemma-cycle}
    \end{lemma}
        
    \noindent
    \textit{Proof of Lemma \ref{lemma-cycle}}: For simplicity assume $\kappa=2$. By Lemma \ref{lemma-index} and a Maslov index calculation,
    \begin{align}
        \mathrm{ind}&(u;\boldsymbol{\Theta},\boldsymbol{\Xi})=(n-2)(\chi-2)+4n-4,
        \label{ind-formula-sum}\\
        \mathrm{ind}&(u;\boldsymbol{\Theta},\{\Xi_2,\Theta_1\})=(n-2)(\chi-2)+n,
        \label{ind-formula-trivial}\\
        \mathrm{ind}&(u;\boldsymbol{\Theta},\{\Xi_1,\Theta_2\})=(n-2)(\chi-2)+3n-4,
        \label{ind-formula-nontrivial}
    \end{align}
    \noindent
    where for example, $\mathrm{ind}(u;\boldsymbol{\Theta},\boldsymbol{\Xi})$ denotes the Fredholm index of curves $u\colon\dot{F}\to\mathbb{R}\times[0,1]\times\widetilde{W}$ such that
    
    \begin{enumerate}
        \item $du\circ J^\lozenge=J^\lozenge\circ du$,
        \item $u(\partial\dot F)\subset\mathbb{R}\times((\{1\}\times\widetilde{\boldsymbol{a}})\cup(\{0\}\times\widetilde{\boldsymbol{a}}'))$,
        \item As $\pi_\mathbb{R}\circ u$ tends to $+\infty$ (resp. $-\infty$), $\pi_{\widetilde{W}}$ tends to $\boldsymbol{\Theta}$ (resp. $\boldsymbol{\Xi}$), where $\pi_{\mathbb{R}}$ and $\pi_{\widetilde{W}}$ are the projections of $u$ to $\mathbb{R}$ and $\widetilde{W}$.
    \end{enumerate}
    
    We explain (\ref{ind-formula-sum})-(\ref{ind-formula-nontrivial}) now. 
    First, the Maslov index of a closed path over $\widetilde{\gamma}_2'$ and $\widetilde{\gamma}_2$ is $n$ by definition, which implies (\ref{ind-formula-trivial}). 
    To get (\ref{ind-formula-nontrivial}), observe that the projection to $\widetilde{D}$ of the curve from $\Theta_1$ to $\Xi_1$ is over the region surrounded by $\widetilde{\gamma}_1'$ and $\widetilde{\gamma}_1$, which contains two critical values of the Lefschetz fibration. 
    The region surrounded by $\widetilde{\gamma}_1'$ and $\widetilde{\gamma}_1$ can be viewed as the outcome after applying Lagrangian surgery twice on a trivial strip to incorporate the two critical values. 
    By Theorem 55.5 of \cite{fukaya2010lagrangian}, each Lagrangian surgery increases the Fredholm index by $n-2$. Therefore, $\mathrm{ind}(u;\boldsymbol{\Theta},\{\Xi_1,\Theta_2\})=(n-2)(\chi-2)+3n-4$ by comparing with (\ref{ind-formula-trivial}). Finally, (\ref{ind-formula-sum}) follows by adding the Maslov index terms of (\ref{ind-formula-trivial}) and (\ref{ind-formula-nontrivial}). 
    
    For $n=2$, the Fredholm index does not depend on $\chi$: $\mathrm{ind}(u;\boldsymbol{\Theta},\boldsymbol{\Xi})=4$ and $\mathrm{ind}(u;\boldsymbol{\Theta},\{\Xi_2,\Theta_1\})=\mathrm{ind}(u;\boldsymbol{\Theta},\{\Xi_1,\Theta_2\})=2$. 
    
    For $n>3$, observe that in all 3 cases the domain $F$ has 2 punctures and thus $\chi$ is even. 
    In particular, $\mathrm{ind}(u;\boldsymbol{\Theta},\boldsymbol{\Xi})$ is even and $\mathrm{ind}(u;\boldsymbol{\Theta},\boldsymbol{\Xi})\neq1$; If $\mathrm{ind}(u;\boldsymbol{\Theta},\{\Xi_1,\Theta_2\})=1$, then $\chi=2-\frac{n-1}{n-2}$, which is not a integer for $n>3$; If $\mathrm{ind}(u;\boldsymbol{\Theta},\{\Xi_2,\Theta_1\})=1$, then $\chi=-1-\frac{1}{n-2}$, which is also not a integer for $n>3$. 
    
    Therefore, for $n=2$ or $n>3$, there is no index 1 curve from $\boldsymbol{\Theta}$, i.e., $\boldsymbol{\Theta}$ is a cocycle in $\widehat{CF}(\widetilde{W},\widetilde{\boldsymbol{a}}',\widetilde{\boldsymbol{a}})$. The case of $\boldsymbol{\Xi}$ is similar.
    \qed
    
    \begin{remark}
        Lemma \ref{lemma-cycle} does not deal with the case of $n=3$, where the index could be 1 for some $\chi$. Therefore, to show $\boldsymbol{\Theta}$ is a cocycle, we still need to understand the index 1 pseudoholomorphic curves from $\Theta_2$ to $\Xi_2$, whose projection to $\widetilde{D}$ is the thin strip surrounded by $\widetilde{\gamma}'_2$ and $\widetilde{\gamma}_2$, and also the curves from $\Theta_1$ to $\Xi_1$, whose projection to $\widetilde{D}$ is the fat strip surrounded by $\widetilde{\gamma}'_1$ and $\widetilde{\gamma}_1$. 
    \end{remark}
    
    It remains to show that $\Psi\circ\Phi$ induces identity on cohomology (with some nonzero coefficient). The composition of $\Psi$ and $\Phi$ is given by the left-hand side of Figure \ref{degenerate}, which is viewed as a degeneration of a family of curves. 
    The right-hand side of Figure \ref{degenerate} shows another degeneration. Observe that the right-hand part of the right degeneration is of index 0. 
    In fact the right degeneration corresponds to the identity map times the count of its left-hand part, which is chain homotopic to $\Psi\circ\Phi$.
    
    \begin{figure}[ht]
        \centering
        \includegraphics[height=4cm]{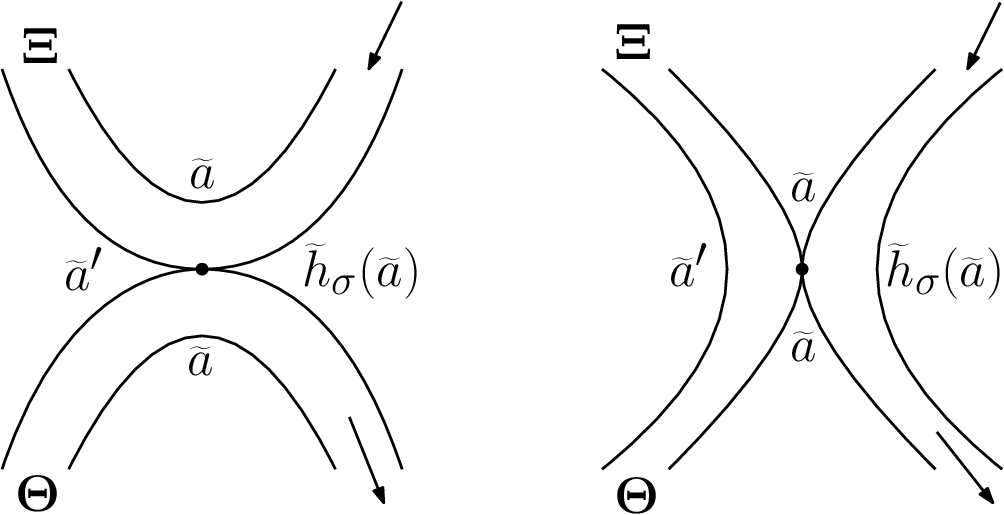}
        \caption{Two possible degenerations.}
        \label{degenerate}
    \end{figure}
    
    Now we count the left-hand part of the right degeneration of Figure \ref{degenerate}. For convenience, assume $\kappa=2$, i.e., focus on the arc sliding of $\widetilde{\gamma}_1$ over $\widetilde{\gamma}_2$. 
    The moduli space $\mathcal{M}_{J}(\boldsymbol{\Xi},\boldsymbol{\Theta})$ contains curves $u$ with boundary condition which maps $\partial\dot F$ to $(\mathbb{R}\times\{1\}\times\boldsymbol{\widetilde{a}'})\cup(\mathbb{R}\times\{0\}\times\boldsymbol{\widetilde{a}})$. 
    There is another restriction on the right degeneration: $u$ passes through $(0,0,w_1)$ and $(0,0,w_2)$, where $(0,0)\in\mathbb{R}\times[0,1]$ and $w_i\in\boldsymbol{\widetilde{a}_i}$. Passing through a generic $\boldsymbol{w}=\{w_1,w_2\}$ is a codimension $2n$ condition. 
    By Lemma \ref{lemma-index}, $\mathrm{ind}(\boldsymbol{\Xi},\boldsymbol{\Theta})=2n$ if and only if $n=2$ or $\chi=0$ for $n>3$.
    
    By Theorem \ref{theorem-full} below, the count of $\mathcal{M}_{J}^{\chi=0}(\boldsymbol{\Xi},\boldsymbol{\Theta})$ passing a generic $\boldsymbol{w}$ is 1 (mod 2). Thus $\Psi\circ\Phi$ is cochain homotopic to identity with some nonzero coefficient. The case of $\Phi\circ\Psi$ is similar.
\end{proof}

\subsection{Half of curve counting}
There is a half version of the counting problem, where the base is shown as Figure \ref{half-base}. The base $D$ can be extended to $\mathbb{C}$ by adding cylindrical ends, denoted by $\bar{D}$. 
Let $\bar{a}_i$ and $\bar{a}_j'$ be the cylindrical completion of $a_i$ and $a_j'$. 
The asymptotic Reeb chords from $\bar{a}_i$ to $\bar{a}_j'$ form a $S^{n-1}$-family. 
We then perturb the contact form so that the $S^{n-1}$-family becomes Morse-Bott and let $\check{c}_{ij},\hat{c}_{ij}$ be the longer and shorter asymptotic Reeb chords. 
Let $\boldsymbol{c}=\{\check{c}_{12},\check{c}_{21}\}$. The problem is to count $\mathcal{M}_{J^{\lozenge}}^{\chi=1,\boldsymbol{w}}(\boldsymbol{c},\boldsymbol{\Theta})$, which is over the cylindrical extension of the region surrounded by $\gamma_1,\gamma_2,\gamma_1',\gamma_2'$. 

\begin{figure}[ht]
    \centering
    \includegraphics[height=4cm]{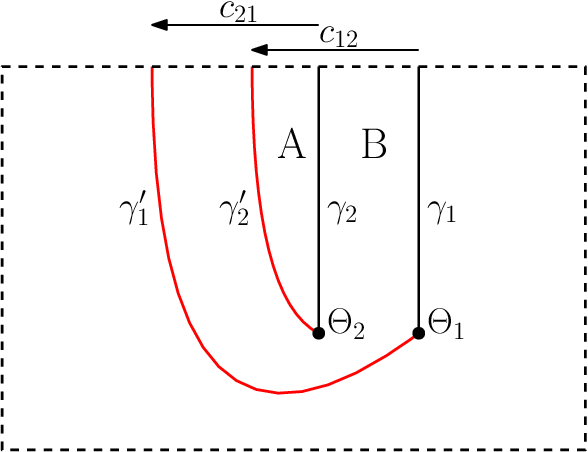}
    \caption{Half of the base. We count pseudoholomorphic disks surrounded by $\gamma_1,\gamma_2,\gamma_1',\gamma_2'$.}
    \label{half-base}
\end{figure}

\begin{theorem}
    $\#\mathcal{M}_{J^{\lozenge}}^{\chi=1,\boldsymbol{w}}(\boldsymbol{c},\boldsymbol{\Theta})=1$ mod $2$ for generic $\boldsymbol{w}$.
    \label{theorem-half}
\end{theorem}
\begin{proof}
    For $u\in\mathcal{M}_{J^{\lozenge}}^{\chi=1,\boldsymbol{w}}(\boldsymbol{c},\boldsymbol{\Theta})$, let $v=\pi_{\bar{W}}\circ u\colon\dot{F}\to \bar{W}$, where $\bar{W}$ is the cylindrical extension of $W$ and $F$ is the unit disk in $\mathbb{C}$.
    
    Recall that $p\colon\bar{W}\to\bar{D}$ is the projection. 
    The map $p\circ v\colon\dot{F}\to \bar{D}$ has degree 2 (resp. 1) over region A (resp. B) of Figure \ref{half-base}, which are connected components of $D-{\gamma}_1\cup{\gamma}'_1\cup{\gamma}_2\cup{\gamma}'_2$ with extension. 
    There are two possible branching behaviors: Type $int$ has a branch point $b$ that maps to the interior of region A; Type $\partial$ is more obscure which has two switch points $b_1,b_2$ on the boundary of region A instead of a branch point. 
    Denote Type $\partial_1$ for the case $b_1,b_2$ on $\gamma_2'$ and Type $\partial_2$ for $b_1,b_2$ on $\gamma_2$. Choose $b_1$ as the point closer to the puncture $c_{12}$ (resp. $c_{21}$) on $\partial \dot{F}$ that maps to $\gamma_2'$ (resp. $\gamma_2$).
    
    There is another codimension-1 constraint. 
    Denote the preimage of $*$ under $u$ as $q(*)$, and we will often leave this notation out if there is no ambiguity.
    A possible arrangement of points on $\partial \dot{F}$ is shown as Figure \ref{unitdisk}.
    Since the projection of $u$ to $\mathbb{R}\times[0,1]$ is a double branched cover and under this projection, $\{q(w_1),q(w_2)\}$, $\{q(\Theta_1),q(\Theta_2)\}$, $\{q(\check{c}_{12}),q(\check{c}_{21})\}$ are mapped to $(0,0),-\infty,+\infty$, respectively, the deck transformation requires the involution:
    \begin{equation}
        q(\Theta_1)\mapsto q(\Theta_2),\,q(\check{c}_{21})\mapsto q(\check{c}_{12}),\,q(w_2)\mapsto q(w_1).
    \label{involution}
    \end{equation}
    
    \begin{figure}
        \centering
        \includegraphics[width=3.5cm]{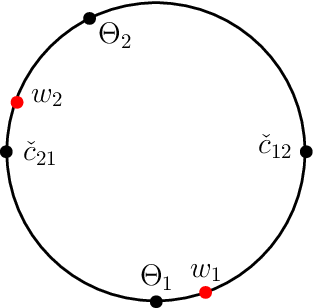}
        \caption{}
        \label{unitdisk}
    \end{figure}
    
    The counting strategy uses the fact that the mod 2 count $\#\mathcal{M}_{J^{\lozenge}}^{\chi=1,\boldsymbol{w}}(\boldsymbol{c},\boldsymbol{\Theta})$ does not depend on $\boldsymbol{w}$, which allows us to stretch the curve by letting $\boldsymbol{w}$ tend to some limit: 
    
    \begin{enumerate}
        \item Let $|p(w_2)|\gg0$; 
        \item Let the width between $\gamma_2,\gamma_2', m\to0$;
        \item Choose $w_1$ such that $p(w_1)\to z_1$.
    \end{enumerate}
    
    \noindent
    $Step$ 1. Denote $\iota=\mathrm{Im}\circ p$. We first observe that as $p(w_1)\to z_1$, either $\iota\circ v(b)\gg \iota(w_2)$ or $\iota\circ v(b_2)\gg \iota(w_2)$. If this is not true, then by Gromov compactness there exists a limiting curve which contradicts the involution condition \ref{involution}. 
    We refer the reader to \cite{colin2020applications} for details.

    \vskip.15in    
    \noindent
    $Step$ 2. Write $\mathcal{M}_{J^{\lozenge}}^{\chi=1,\boldsymbol{w},\sharp}(\boldsymbol{c},\boldsymbol{\Theta})$ for the curves satisfying $\iota\circ v(b)\gg\iota(w_2)$ or $\iota\circ v(b_1)\geq\iota(w_2)-C$. 
    The case of $\iota\circ v(b)\gg\iota(w_2)$ is shown as the left of Figure \ref{half-split}, where the limiting 2-level curve contains $v^{(1)}\cup v^{(2)}\cup v^{(3)}$. The case of $\iota\circ v(b_1)\geq\iota(w_2)-C$ is similar. We push off the bottom singular point a little so that $a_1$ and $a_1'$ have clean $S^{n-1}$-intersection, which does not change the moduli space of $v^{(1)}$ by \cite{fukaya2010lagrangian}.

    \begin{figure}[ht]
        \centering
        \includegraphics[height=6cm]{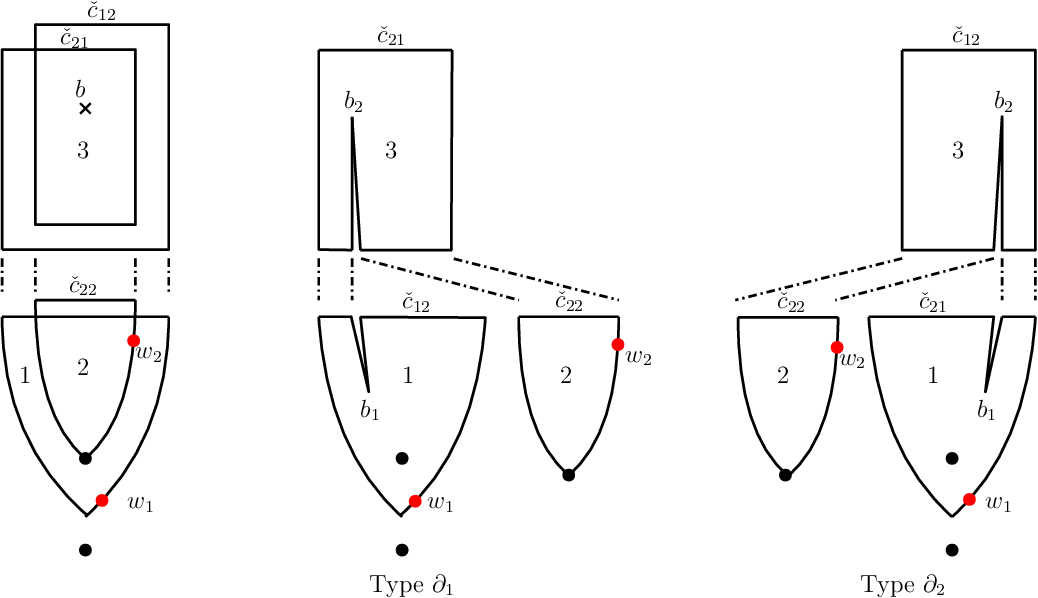}
        \caption{}
        \label{half-split}
    \end{figure}
    
    The region of $p\circ v^{(1)}$ contains a singular point of a Lefschetz fibration, which can be viewed as the outcome after a Lagrangian surgery from the trivial region without singular points. By Theorem 55.5 of \cite{fukaya2010lagrangian}, the moduli space of $v^{(1)}$ is diffeomorphic to $S^{n-2}$. Therefore the evaluation map of $v^{(1)}$ at $\check{c}_{11}$ sweeps $S^{n-2}$ inside the $S^{n-1}$-family of Reeb chords, of which the homology class vanishes. The result is $\#\mathcal{M}_{J^{\lozenge}}^{\chi=1,\boldsymbol{w},\sharp}(\boldsymbol{c},\boldsymbol{\Theta})=0$ mod 2.
    
    \vskip.15in
    \noindent
    $Step$ 3. The remaining case is the curve satisfying $\iota\circ v(b_2)\gg\iota(w_2)$ and $\iota\circ v(b_1)\leq\iota(w_2)-C$. The stretched limiting curves are shown as the middle and the right of Figure \ref{half-split}. We expect that such curves exist and contribute 1 (mod 2) to $\mathcal{M}_{J^{\lozenge}}^{\chi=1,\boldsymbol{w}}(\boldsymbol{c},\boldsymbol{\Theta})$.
    
    First we treat Type $\partial_1$. $v^{(2)}$ is uniquely determined since there is a unique gradient trajectory from $\check{c}_{22}$ to $\Theta_2$ passing through $w_2$. Denote the bottom-left Reeb chord of $v^{(3)}$ by $\hat{d}_{21}$. We show that the pseudoholomorphic triangle $v^{(3)}$ exists uniquely: The Lefschetz fibration around $v^{(3)}$ is the trivial one $p\colon\mathbb{C}\times T^*S^{n-1}\to \mathbb{C}$. 
    We can perturb the Lagrangians $a_1,a_1',a_2,a_2'$ in the fiber direction. The case of $n=2$ is shown as Figure \ref{perturb-fiber}, 
    where there exists a unique pseudoholomorphic triangle.
    
    \begin{figure}[ht]
        \centering
        \includegraphics[height=4cm]{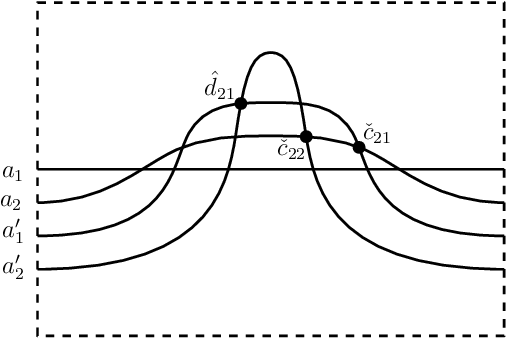}
        \caption{The fiber $T^*S^{1}$ with perturbed Lagrangians. The sides are identified.}
        \label{perturb-fiber}
    \end{figure}
    
    More generally, the case of $n\geq 2$ is done by the gradient tree argument. Put the perturbed Lagrangians as Figure \ref{perturb-fiber} on $T^*S^1\subset T^*S^{n-1}$ and extend to $T^*S^{n-1}$, where we view $S^1\subset S^{n-1}$ as a equator.
    After a further small perturbation, we consider the moduli space of Morse gradient trees with 3 vertices $\check{c}_{22},\check{c}_{21},\hat{d}_{21}$. Viewing all vertices as sources of gradient flow, we check that $\mathrm{ind}(\check{c}_{22})=0$, $\mathrm{ind}(\check{c}_{21})=\mathrm{ind}(\hat{d}_{21})=n-1$. 
    
    By Lemma \ref{lemma-morse}, the dimension of this moduli space is 0 and it contains a unique gradient tree which corresponds to a single pseudoholomorphic triangle by Theorem \ref{Fukaya-Oh}. Note that this is almost the same as Example \ref{ex-tree}.
     
    It remains to consider $v^{(1)}$, which is a pseudoholomorphic disk with a slit and a singular point inside. 
    The singular point can be viewed as coming from a Lagrangian surgery \cite{fukaya2010lagrangian}, which increases the Fredholm index by $n-2$. 
    Now the moduli space of $v^{(1)}$ passing through a generic $w_1$ is of dimension $n-1$.
    Consider the evaluation map $ev_{21}$ of $v^{(1)}$ on its top-left end, where the image lies in a $S^{n-1}$-family of Reeb chords. The gluing condition on both ends says $ev_{21}$ should take the value of $\hat{d}_{21}$. 
    All we need is that $ev_{21}$ intersects $\hat{d}_{21}$ at a unique point. We will show in Type $\partial_1$, $ev_{21}$ sweeps half of $S^{n-1}$ and the other half is dealt with by Type $\partial_2$. This is done by a model calculation of explicit pseudoholomorphic curves in Step 3' below.

    \vskip.15in
    \noindent
    $Step$ 3'. \textbf{A model calculation}. All notations are limited to this step. 
    
    We replace the base of $v^{(1)}$ in Type $\partial_1$ by a standard one, i.e., the unit disk in $\mathbb{C}_z$ with a slit $\{-1\leq\mathrm{Re}\,z\leq0\}\cap\{\mathrm{Im}\,z=0\}$. The Lefschetz fibration over the unit disk is
    \begin{equation}
        p\colon(z_1,z_2,\dots,z_n)\mapsto z_1^2+z_2^2+\dots+z_n^2,
    \label{lefschetz}
    \end{equation}
    with a critical value at $0\in\mathbb{C}_z$. It is however more convenient to think of the case $n=2$ first, and the Lefschetz fibration is
    \begin{equation}
        p'\colon(z_1',z_2')\mapsto z_1'z_2',
    \label{clifford}
    \end{equation}
    where $z_1'=z_1+iz_2,\,z_2'=z_1-iz_2$. Let $T$ be the Clifford torus $\{|z_1'|=1\}\times\{|z_2'|=1\}$ over $|z|=1$ and let $L$ be the Lagrangian thimble over $\{-1\leq\mathrm{Re}\,z\leq0\}\cap\{\mathrm{Im}\,z=0\}$. $T\cap L$ is a clean $S^1$-intersection over $-1\in\mathbb{C}_z$.
    
    We consider curves with boundary on $T\cup L$: let $\mathcal{M}_2$ (2 stands for $n=2$) be the moduli space of holomorphic disks
    \begin{equation*}
        u=(u'_1,u'_2)\colon\mathbb{R}\times[0,1]\to\mathbb{C}^2_{z_1',z_2'}
    \end{equation*}
    with standard complex structure $J_{std}$ satisfying
    \begin{enumerate}
        \item $u(\mathbb{R}\times\{0\})\subset T$ and $u(\mathbb{R}\times\{1\})\subset L$;
        \item $p'\circ u$ has degree 1 over $\{|z|<1\}-\{-1\leq\mathrm{Re}\,z\leq0\}\cap\{\mathrm{Im}\,z=0\}$ and degree 0 otherwise;
        \item $u(0,0)=w_1=(w_{11}',w_{12}')=(1,1)$.
    \end{enumerate}

    \begin{figure}
        \centering
        \includegraphics[width=12cm]{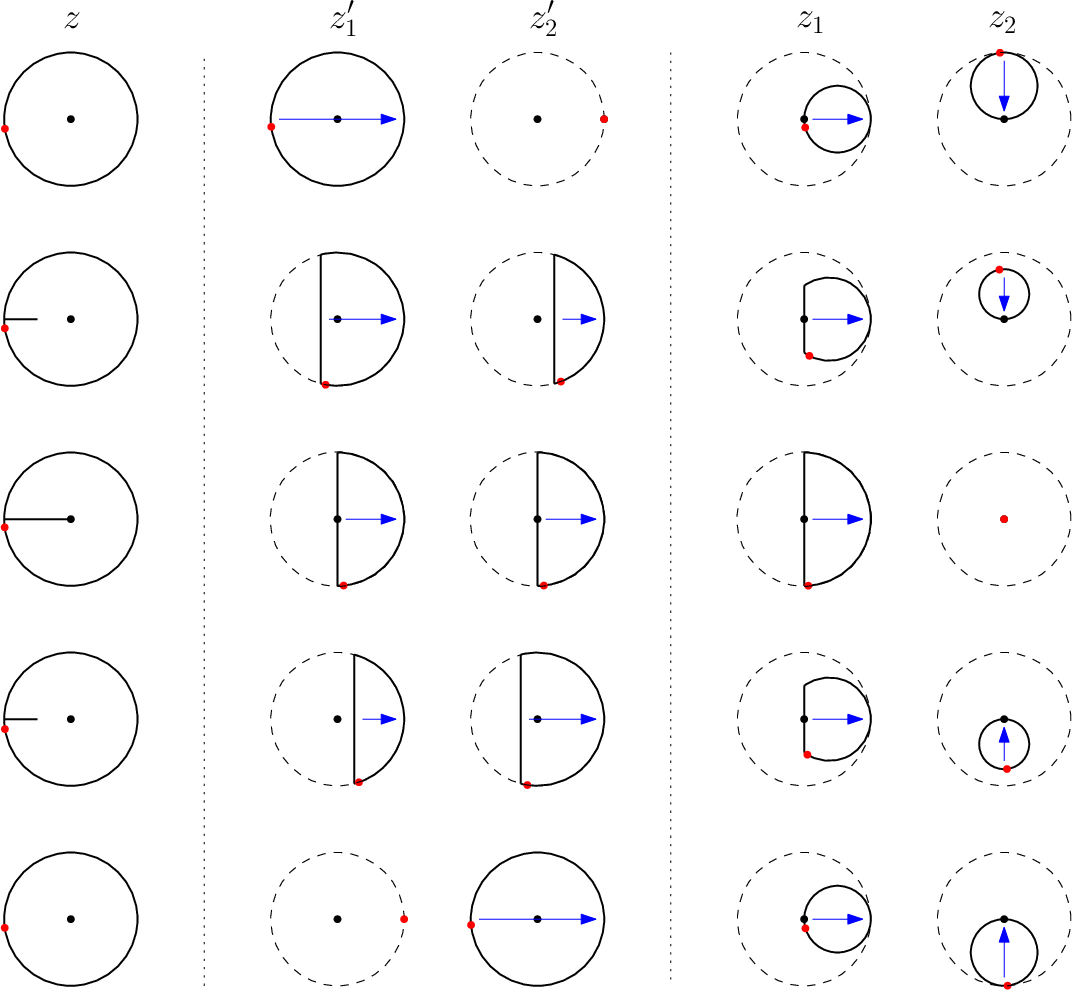}
        \caption{The pictorial description of $\mathcal{M}_2$. The left column describes the base $\mathbb{C}_z$. The middle and right column are for two sets of coordinates. Dashed curves denote the unit circle. The red dots indicate $e^{i(\pi+\epsilon)}$ on the left and for their preimages on the middle and right. The arrows indicate images of the arrow from $(0,1)$ to $(0,0)$ on the domain $\mathbb{R}\times[0,1]$. Note that the drawings are not necessarily accurate.}
        \label{disk-moduli}
    \end{figure}
    
    Condition (3) is essentially the same as that $v^{(1)}$ passes through $w_1$ in Step 3. It is not hard to see that $\mathcal{M}_2$ is homeomorphic to a line segment where $\partial\mathcal{M}_2$ consists of two curves $z\mapsto(z,1)$ and $z\mapsto(1,z)$. Figure \ref{disk-moduli} gives a schematic description of $\mathcal{M}_2$, from the top row of $z\mapsto(z,1)$ to the bottom row of $z\mapsto(1,z)$, where the right-hand column changes the coordinates to $(z_1,z_2)$.
    
    We then consider the evaluation map. The top-left end of $v^{(1)}$ in Type $\partial_1$ is translated to $e^{i(\pi+\epsilon)}\in\mathbb{C}_z$ for small $\epsilon>0$. 
    Define $ev'_1\colon\mathcal{M}_2\to S^1_{|z'_1|=1}$ as the $z'_1$ projection of the intersection between $u$ and $p'^{-1}(e^{i(\pi+\epsilon)})$, which is shown by red dots in the $z'_1$ column of Figure \ref{disk-moduli}.
    Clearly $ev'_1$ is a homeomorphism between $\mathcal{M}_2$ and $\{e^{i\theta_1}|\pi+\epsilon<\theta_1<2\pi\}$.
    
    For general $n\geq2$, define $\mathcal{M}_n$ similar to $\mathcal{M}_2$, where $T,L$ in condition (1) are still the Lagrangian vanishing cycles over the unit circle and $\{-1\leq\mathrm{Re}\,z\leq0\}\cap\{\mathrm{Im}\,z=0\}$. Condition (3) is modified to
    
    \begin{enumerate}
        \item[(3')] $u(0,0)=w_1=(w_{11},w_{12},\dots,w_{1n})=(1,0,\dots,0)$.
    \end{enumerate}
    
    The moduli space $\mathcal{M}_2$ in coordinate $(z_1,z_2)$ is viewed as the slice of $\mathcal{M}_n$ that restricts to $0$ on $z_3,\dots,z_n$. 
    Observe that for $\mathcal{M}_n$, the coordinates $z_2,z_3,\dots,z_n$ are symmetric. Thus we can recover $\mathcal{M}_n$ from $\mathcal{M}_2$ by a symmetric rotation of $z_2$-coordinate. Each $u=(z_1,z_2)\in\mathcal{M}_2$ corresponds to a $S^{n-2}$-family of curves in $\mathcal{M}_n$:
    \begin{equation*}
        u_{\lambda_1,\dots,\lambda_{n-1}}=(z_1,\lambda_1 z_2,\dots,\lambda_{n-1} z_2),
    \end{equation*}
    where $\lambda_1^2+\dots+\lambda_{n-1}^2=1$ and $\lambda_i\in\mathbb{R}$ for $i=1,\dots,n-1$. Figure \ref{disk-n} shows $u_{1/\sqrt{n-1},\dots,1/\sqrt{n-1}}$ recovered from the first row of Figure \ref{disk-moduli}. 
    
    \begin{figure}[ht]
        \centering
        \includegraphics[width=8cm]{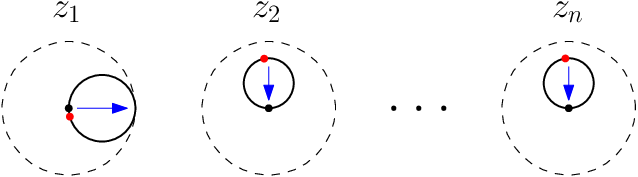}
        \caption{$u_{1/\sqrt{n-1},\dots,1/\sqrt{n-1}}$ of the first row in Figure $\ref{disk-moduli}$.}
        \label{disk-n}
    \end{figure}
    
    The new evaluation map $ev$ is defined as the $n$-tuple of coordinates which projects to $e^{i(\pi+\epsilon)}\in\mathbb{C}_z$. Therefore, $\mathcal{M}_n$ is homeomorphic to $D^{n-1}$. One can check that $ev\colon\mathcal{M}_n\to S^{n-1}$ is a homeomorphism to its image, which is half of the vanishing cycle $S^{n-1}$ over $e^{i(\pi+\epsilon)}$.
    
    In case $J_{std}$ is not regular, we apply small perturbation $J^{\lozenge}$ of $J_{std}$. One can show that for any $\boldsymbol{z}\in ev_{J_{std}}(\mathcal{M}_n)$, $\#ev^{-1}_{J^\lozenge}(\boldsymbol{z})=1$ mod $2$. Therefore the argument of Step 3 still works. This finishes Step 3'.
    
    \vskip.15in
    \noindent
    Finally we glue $v^{(1)},v^{(2)}$ and $v^{(3)}$. Still assume we are in Type $\partial_1$. The involution condition (\ref{involution}) will fix the neck length: As we take $\iota\circ v(b_2)\to\infty$, $q(\Theta_2)$ approaches $q(\check{c}_{21})$ but $|q(w_2)-q(\Theta_2)|\ll|q(w_2)-q(\check{c}_{21})|$. Assume $w_1$ is close to $\Theta_1$, there is a unique value of $\iota\circ v(b_2)$ for which there exists an involution of $F$. This completes the proof of Theorem \ref{theorem-half}. 
\end{proof}

\subsection{Curve counting}

The goal of this section is to count the full version of pseudoholomorphic annuli from $\Xi_1,\Xi_2$ to $\Theta_1,\Theta_2$ over the region in Figure \ref{arc-slide} and prove the following theorem:
\begin{theorem}
$\#\mathcal{M}^{\chi=0,\boldsymbol{w}}_{J^{\lozenge}}(\boldsymbol{\Xi},\boldsymbol{\Theta})=1$ mod $2$ for generic $\boldsymbol{w}$.
\label{theorem-full}
\end{theorem}

\noindent
\textit{Proof}. The strategy is the same as the proof of Theorem \ref{theorem-half}: We stretch the curve into several levels by choosing some extreme $\boldsymbol{w}$, and then use the restriction of domain involution and gluing conditions to find a unique (mod $2$) curve. 

We closely follow the proof of Theorem 9.3.7 of \cite{colin2020applications} and some details are omitted. As before, we write $u\colon\dot{F}\to\mathbb{R}\times[0,1]\times\widetilde{W}$ for an element in $\mathcal{M}^{\chi=0,\boldsymbol{w}}_{J^{\lozenge}}(\boldsymbol{\Xi},\boldsymbol{\Theta})$ and let $v$ be its projection to $\widetilde{W}$. 

The main idea is to stretch the base $\widetilde{D}$ in $\mathrm{Im}\,z$ direction as Figure \ref{full-split}: Let $\mathrm{Im}\,z_i=-2K$ and $\mathrm{Im}\,z_{i+\kappa}=2K$, $i=1,\dots,\kappa$ and $K\to+\infty$. 
The region $\mathcal{R}$ bounded by $\gamma_1,\gamma_1'$ is split into 3 parts: $\mathcal{R}_1=\mathcal{R}\cap\{\mathrm{Im}\,z\leq-K\}$, $\mathcal{R}_2=\mathcal{R}\cap\{-K\leq\mathrm{Im}\,z\leq K\}$ and $\mathcal{R}_3=\mathcal{R}\cap\{\mathrm{Im}\,z\geq K\}$. 
The mod $2$ count is independent of $\boldsymbol{w}$. 
We choose $\widetilde{w}_1$ close to $z_1$, $\widetilde{w}_2$ close to $\mathrm{Im}\,z=0$ and the thin strip $\mathcal{R}'$ between $\gamma_2, \gamma_2'$ with width $m\to 0$.

\begin{figure}[ht]
    \centering
    \includegraphics[height=2.5cm]{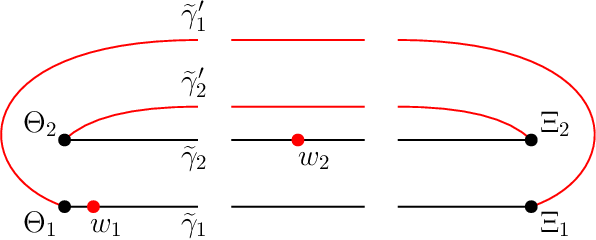}
    \caption{The stretched base as $K\gg0$. $\mathrm{Im}\,z$ is the horizontal direction.}
    \label{full-split}
\end{figure}

The curve $v$ has degree 2 over $\mathcal{R}'$ and degree 1 over $\mathcal{R}-\mathcal{R}'$. The types of branching behaviors are denoted by $2$, $1L$, $1R$, $0LL$, $0LR$, $0RR$. Take $0LR$ for example: 0 means the number of interior branch points is 0; $L$ means one pair of switch points is over $\widetilde{\gamma}_2'$; $R$ means one pair of switch points is over $\widetilde{\gamma}_2$. Denote (if exist) interior branch points by $b,b'\in \mathrm{int}(\dot{F})$ and switch points by $b_1,b_2,b_3,b_4\in\partial\dot{F}$. We assume $\iota(b')>\iota(b)$, $\iota(b_2)>\iota(b_1)$ and $\iota(b_4)>\iota(b_3)$. 

\vskip.15in
\noindent
$Step$ $1$. Suppose $K\gg 0$. Take a sequence of $u^{(i)}\in\mathcal{M}^{\chi=0,\boldsymbol{w}}_{J^{\lozenge}}(\boldsymbol{\Xi},\boldsymbol{\Theta})$ so that $v^{(i)}\colon F^{(i)}\to\widetilde{W}$ with $m^{(i)}\to0$. 
In the limit $i\to\infty$, the thin strip tends to a slit, the limiting curve splits into $v_{\infty}\cup\delta_+\cup\delta_-$, where $\delta_+$ is a gradient trajectory from $\Xi_2$ and $\delta_-$ is a gradient trajectory to $\Theta_2$, and $v_\infty$ is a pseudoholomorphic annulus. 
Figure \ref{full-limit} describes the limiting procedure in the case of Type $0LR$. In fact we will show that Type $0LR$ is the only nontrivial case. 

\begin{figure}[ht]
    \centering
    \includegraphics[width=8cm]{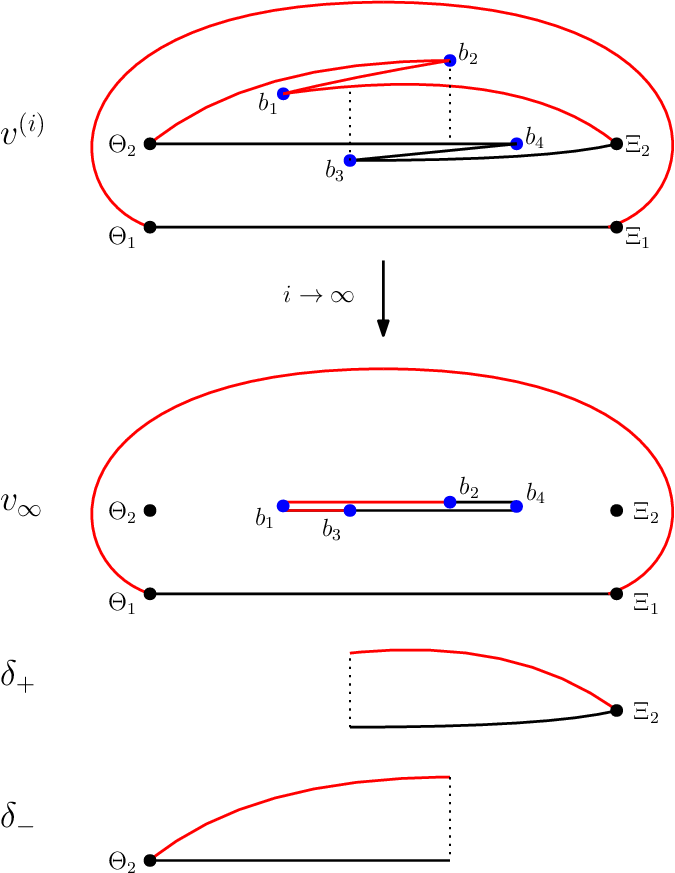}
    \caption{The limiting proceduce of Type $0LR$.}
    \label{full-limit}
\end{figure}

\noindent
$Step$ $2$. We claim that the limiting slit is long enough, i.e., for $K\gg0$ and $m\to0$,
\begin{equation}
    \mathrm{max}\{\iota(b'),\iota(b_2),\iota(b_4)\}\geq K,
\label{max_K}
\end{equation}
\begin{equation}
    \mathrm{min}\{\iota(b),\iota(b_1),\iota(b_3)\}\leq -K.
\label{min_K}
\end{equation}
which are the two endpoints of the slit. 

If (\ref{min_K}) is not true, i.e., $\mathrm{min}\{\iota(b),\iota(b_1),\iota(b_3)\}> -K$, then the part of $v_\infty$ in region $\mathcal{R}_1$ has no slit, which can be viewed as the outcome of a Lagrangian surgery on a trivial pseudoholomorphic disk. Similar to Step 2 in the proof of Theorem \ref{theorem-half}, the moduli space of pseudoholomorphic disks passing through a generic $w_1$ is diffeomorphic to $S^{n-2}$, of which the evaluation map at the cylindrical end has a $S^{n-2}$-intersection with the $S^{n-1}$-family of Reeb chords. The evaluation map vanishes at homology level, which contributes 0 (mod 2) to the curve count. The argument for (\ref{max_K}) is similar.
~\\

\noindent
$Step$ $3$. We claim that for $K\gg0$ and $m$ small, if (\ref{max_K}) and (\ref{min_K}) hold, the mod $2$ contribution of Type $2,1L,0LL,0RR$ is 0. The reason is that if one considers the involution condition
\begin{equation}
    q(\Theta_1)\mapsto q(\Theta_2),\,q(\Xi_1)\mapsto q(\Xi_2),\,q(w_1)\mapsto q(w_2)
\label{full-involution}
\end{equation}
for a pseudoholomorphic annulus, there is a constraint on the position of $\Theta_2,\Xi_2,w_2$ on the slit in $v_\infty$. One can refer to \cite{colin2020applications} for detailed discussion that all but Type $0LR$ contradict with (\ref{full-involution}).

\vskip.15in
\noindent
$Step$ $4$. It remains to consider Type $0LR$. Assuming (\ref{max_K}) and (\ref{min_K}) are satisfied, we claim that the contribution of Type $0LR$ is 1 (mod 2). 

Although there are other possible arrangements of $b_1,b_2,b_3,b_4$ on the slit, we just consider the case in Figure
\ref{full-limit} for illustration. As shown in Figure \ref{full-glue}, $v_\infty$ is the gluing of two regions: $v_{l,\infty}$ with $\mathrm{Im}\,z\ll-K$ and $v_{r,\infty}$ with $\mathrm{Im}\,z\gg-K$, which can be viewed as two pseudoholomorphic disks similar to Step 3' of the previous section. The conditions that $w_1$ is close to $\Theta_1$ and that $w_2$ sits on the slit will be translated to an explicit model calculation in Step 4' below. For suitable choices of $w_1$ and $w_2$, we will show that the space of two disks with $c_{21}=c^*_{21},\,c_{12}=c^*_{12}$ is homeomorphic to a line segment $\mathcal{I}$. Let $M_{\mathcal{I}}$ be the set of $v_\infty$ glued from the $\mathcal{I}$-family of $(v_{l,\infty},v_{r,\infty})$. 

\begin{figure}[ht]
    \centering
    \includegraphics[height=2cm]{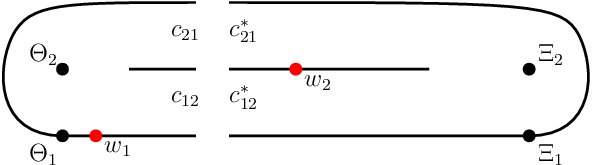}
    \caption{$v_\infty$ with a long slit. $\mathcal{R}_1$ on the left and $\mathcal{R}_2\cup\mathcal{R}_3$ on the right.}
    \label{full-glue}
\end{figure}

The involution condition (\ref{full-involution}) determines a unique curve inside $M_\mathcal{I}$. In conclusion, $\#\mathcal{M}^{\chi=0,\boldsymbol{w}}_{J^{\lozenge}}(\boldsymbol{\Xi},\boldsymbol{\Theta})=1$ mod $2$. We have proved Theorem \ref{theorem-full} modulo the model calculation below:

\vskip.15in
\noindent
$Step$ 4'. \textbf{A model calculation}. We use similar notations as in Step 3' of the previous section. Define $\mathcal{M}_l$ for the space of $v_l$ over the left side of Figure \ref{full-glue} and $\mathcal{M}_r$ for those of $v_r$ over the right side. Both are viewed as maps from $\mathbb{R}\times[0,1]$ to $\mathbb{C}^n$ over the unit circle with one slit $\subset\{-1\leq\mathrm{Re}\,z\leq0\}\cap\{\mathrm{Im}\,z=0\}\subset\mathbb{C}_z$. Consider the evaluation maps
\begin{equation*}
    ev_l,ev_r\colon\mathcal{M}_l,\mathcal{M}_r\to S^{n-1}\times S^{n-1},
\end{equation*}
which are defined below, corresponding to $c_{12},c_{21},c^*_{12},c^*_{21}$ in Figure \ref{full-glue}. The gluing condition is $ev_l(v_l)=ev_r(v_r)$. 

As usual, we first consider the case of $n=2$ with Lefschetz fibration (\ref{clifford}). 

The moduli space $\mathcal{M}_l$ is defined as holomorphic disks passing through $w_1=(1,1)$ over $1\in\mathbb{C}_z$. Note that $\mathcal{M}_l$ is the same as $\mathcal{M}$ in Step 3' of the previous section. Let $w_{\pm}=e^{i(\pi\pm\epsilon)}\in\mathbb{C}_z$. Define the evaluation map $ev_l$ as
\begin{gather*}
    ev_l\colon\mathcal{M}_l\to S^1\times S^1,\\
    v_l\mapsto(ev_{l+}(v_l),ev_{l-}(v_l)), 
\end{gather*}
where $ev_{l\pm}(v_l)$ is the $z_1'$-coordinate of the point that projects to $w_{\pm}$. Figure \ref{1-moduli} gives a schematic description of $\mathcal{M}_l$, homeomorphic to a line segment, and its evaluation maps $ev_{l\pm}$ denoted by red and violet dots. 
We also show the maps in coordinates $(z_1,z_2)$ with $z_1'=z_i+iz_2,\,z_2'=z_i-iz_2$. 

\begin{figure}[ht]
    \centering
    \includegraphics[width=12cm]{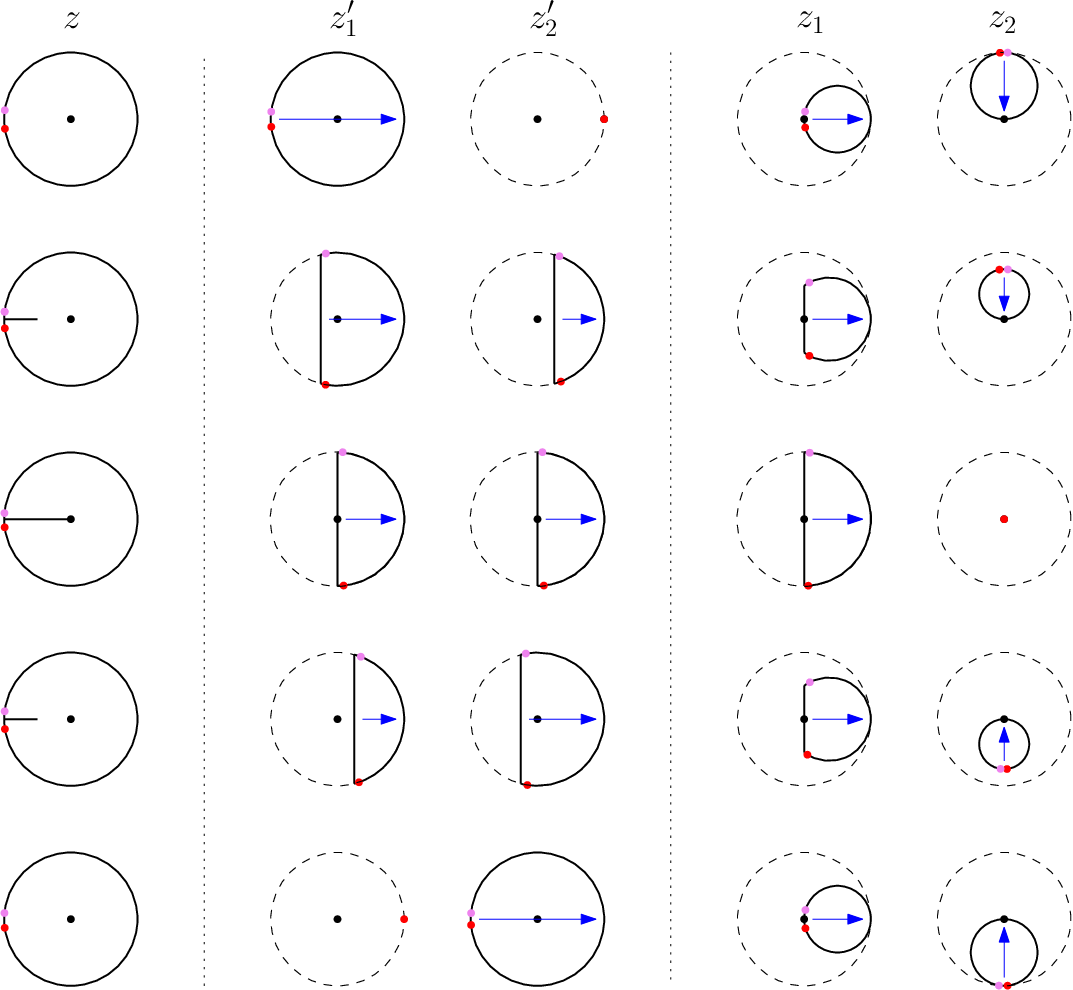}
    \caption{The pictorial description of $\mathcal{M}_l$. The left column describes the base $\mathbb{C}_z$. The middle and right columns are for two sets of coordinates. The red and violet dots indicate $e^{i(\pi\pm\epsilon)}$ on the left and their preimages on the middle and right.}
    \label{1-moduli}
\end{figure}

Since $w_1=(1,0)$ in coordinate $(z_1,z_2)$, define $w_1=(1,0,0,\dots,0)$ for $n\geq2$. Then $\mathcal{M}_{l,2}$, the moduli space for $n=2$, is viewed as the slice of $\mathcal{M}_{l,n}$ that restricts to $0$ on $z_3,\dots,z_n$. Now $\mathcal{M}_{l,n}$ is a symmetric rotation of $\mathcal{M}_{l,2}$ which contains 
\begin{equation*}
    v'=(z_1,\lambda_1 z_2,\dots,\lambda_{n-1} z_2),
\end{equation*}
where $\lambda_1^2+\dots+\lambda_{n-1}^2=1$, $\lambda_1,\dots,\lambda_{n-1}\in\mathbb{R}$. Thus $\mathcal{M}_{l,2}$ is homeomorphic to $D^{n-1}$.

Next we consider $\mathcal{M}_r$ of $v''\colon\mathbb{R}\times[0,1]\to\mathbb{C}^n$ and first let $n=2$.
We put the constraint that $v_r$ passes through $(-r,r)\in L$ for some $r\in[0,1]$, which corresponds to $w_2$ sitting on the slit in Figure \ref{full-glue}. 
Let $ev_{r\pm}(v_r)$ be the $z_1'$-coordinate of the point that projects to $w_\pm$ and define the map
\begin{gather*}
    ev_r\colon\mathcal{M}_r\to S^1\times S^1,\\
    v_r\mapsto(ev_{r-}(v_r),ev_{r+}(v_r)), 
\end{gather*}
where $+$ and $-$ are switched because we want to identify $w_\pm$ of $v_l$ with $w_\mp$ of $v_r$.

\begin{figure}[ht]
    \centering
    \includegraphics[width=13cm]{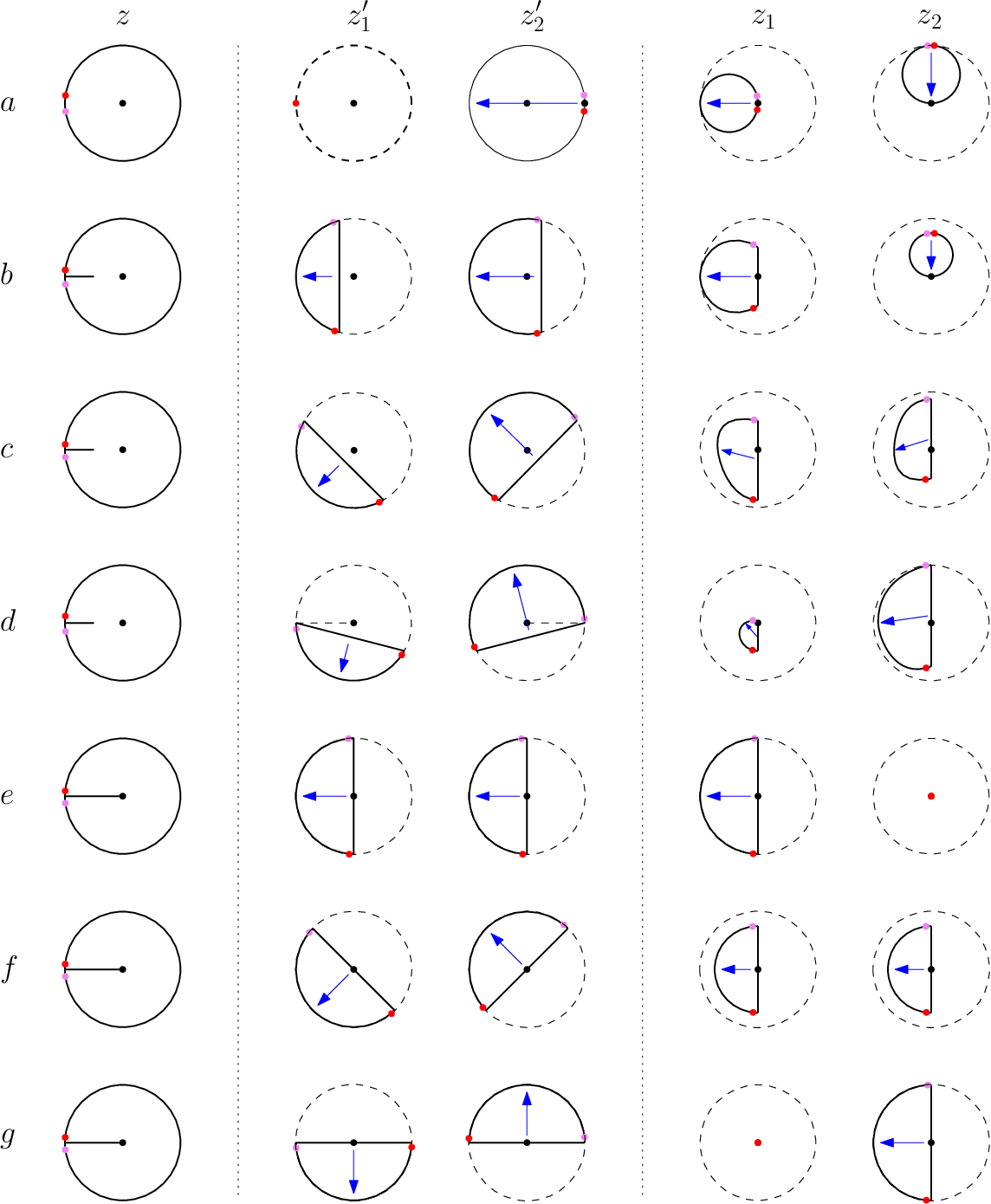}
    \caption{The pictorial description of part of $\mathcal{M}_r$. The notations are as before, while the red and violet dots indicate $e^{i(\pi\mp\epsilon)}$ on the left and their preimages on the middle and right.}
    \label{2-moduli}
\end{figure}

Observe that $\mathcal{M}_{r,2}$ is of dimension 2. Figure \ref{2-moduli} describes some of the curves inside $\mathcal{M}_{r,2}$. $(-r,r)$ in $(z_1',z_2')$ equals $(0,ir)$ in $(z_1,z_2)$. Thus for $n\geq2$, if $\mathcal{M}_{r,2}$ is viewed as the slice of $\mathcal{M}_{r,n}$ with $z_3=\dots=z_n=0$, then $\mathcal{M}_{r,n}$ contains curves of 
\begin{equation*}
    v''=(\lambda_1 z_1,z_2,\lambda_2 z_1,\dots,\lambda_{n-1} z_1),
\end{equation*}
where $\lambda_1^2+\dots+\lambda_{n-1}^2=1$, $\lambda_1,\dots,\lambda_{n-1}\in\mathbb{R}$. 

For $n=2$, we have the following observation:
\begin{claim}
For $n=2$ and $\epsilon\to0$, the evaluation map of $\mathcal{M}_{l}$ and $\mathcal{M}_{r}$ with images in $S^1\times S^1$ is described in Figure \ref{evaluation}. $ev_l(\mathcal{M}_{l})$ is the blue line segment and $ev_r(\mathcal{M}_{r})$ is the 2-dimensional pink region. Their intersection is a line segment $\mathcal{I}$.
\end{claim}

\begin{figure}[ht]
    \centering
    \includegraphics[width=5cm]{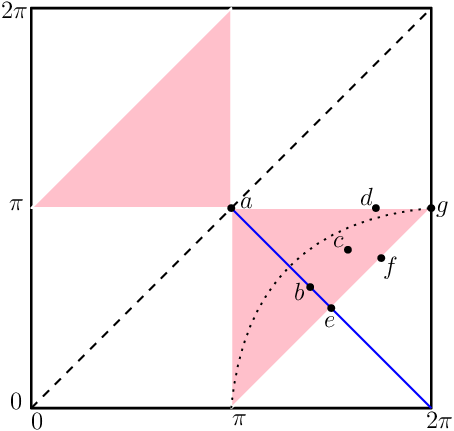}
    \caption{The description of $ev_l$ (blue) and $ev_r$ (pink) for $n=2$. The sides are identified. $a$-$g$ correspond to curves in Figure \ref{2-moduli}.}
    \label{evaluation}
\end{figure}

For general $n\geq2$, we have the same result:
\begin{lemma}
    For $n\geq2$ and $\epsilon\to0$, the intersection between $ev_l(\mathcal{M}_{l})$ and $ev_r(\mathcal{M}_{r})$ is still $\mathcal{I}$.
    \label{high-intersection}
\end{lemma}
\noindent
\textit{Proof of Lemma \ref{high-intersection}}. Suppose $v_{l,\boldsymbol{\lambda}}=(z_{l1},\lambda_1z_{l2},\dots,\lambda_{n-1}z_{l2})$ and $v_{r,\boldsymbol{\lambda'}}=(\lambda'_1 z_{r1}, \linebreak z_{r2},\lambda'_2 z_{r1},\dots,\lambda'_{n-1}z_{r1})$ satisfy $ev_l(v_{l,\boldsymbol{\lambda}})=ev_r(v_{r,\boldsymbol{\lambda'}})$ where the 3rd to $n$-th coordinates are not all zero. Denote $ev_{l\pm}(v_{l,\boldsymbol{\lambda}})=(z_{l1\pm},\lambda_1z_{l2\pm},\dots,\lambda_{n-1}z_{l2\pm})$ and $ev_{r\mp}(v_{r,\boldsymbol{\lambda'}})=(\lambda'_1 z_{r1\mp},z_{r2\mp},\lambda'_2 z_{r1\mp},\dots,\lambda'_{n-1}z_{r1\mp})$. 

Observe that $\mathrm{Im}\,\lambda_1z_{l2+}=\mathrm{Im}\,\lambda_1z_{l2-}$ and then $z_{r2-}=z_{r2+}$. From $c$ and $d$ in Figure \ref{2-moduli} we see that $ev_r((z_{r1},z_{r2}))$ (in coordinate $(z_1',z_2')$) must lie in $\{\theta_1+\theta_2=2\pi\}\cap\{\pi\leq\theta_1\leq3\pi/2\}$ of Figure \ref{evaluation}. For such curves $v$ (as ($a,b,e$) in Figure \ref{2-moduli}), $\mathrm{Im}\,z_{r1-}=-\mathrm{Im}\,z_{r1+}$ and $\mathrm{Re}\,z_{r1-}=\mathrm{Re}\,z_{r1+}$. 
Since $\mathrm{Re}\,z_{l2+}=-\mathrm{Re}\,z_{l2-}$ and $\mathrm{Im}\,z_{l2+}=\mathrm{Im}\,z_{l2-}$, the consequence is that $\lambda_2z_{l2\pm}=\dots=\lambda_{n-1}z_{l2\pm}=0$, which is a contradiction. 
\qed

\vskip.15in
\noindent
Now we go back to the curve counting problem. We want to pick a single curve from the $\mathcal{I}$-family of $v_\infty$ and then glue it to get a unique $v^{(i)}$ for each $i$.

First we show the position of $b_2^\infty$ determines $v_\infty$ uniquely in $\mathcal{I}$: Consider the slit in $v_\infty$ of Figure \ref{full-limit}. In the limit $q(w_1)=q(\Theta_1)$, so $q(w_2)=q(\Theta_2)=q(b_2^\infty)$. If we fix $b_2^\infty$, then $v_\infty''$ passes through $(-r,r)$ for some $r\in[0,1]$, corresponding to fixing a hypersurface in $S^{n-1}\times S^{n-1}$, whose intersection with $S^1\times S^1$ is the dotted arc in Figure \ref{evaluation}. 
The dotted arc intersects the blue line at a single point, which determines $v_\infty$. 
Moreover, the length of the slit in $v_\infty'$ and $v_\infty''$ are determined and thus $b_1^\infty$ and $b_4^\infty$ are fixed. Finally $b_3^\infty$ is fixed by the involution of $F^\infty$.

Consider then $v^{(i)}$ for large $i$. From the previous paragraph $\iota(b_2^{(i)})$ will fix a unique $v_\infty$ in $\mathcal{I}$. By Implicit Function Theorem, it will fix a unique $v^{(i)}$ as well, which is close to $v_\infty$. Then observe that the distance between $q(\Theta_2)$ and $q(w_2)$ is a monotone function of $\iota(b_2^{(i)})$: As $\iota(b_2^{(i)})$ increases, the slit gets longer, $b_1^{(i)}$ moves left and $b_4^{(i)}$ moves right. Therefore $q(w_2)$ leaves $q(\Theta_2)$ and approaches $q(\Xi_2)$ on $F^{(i)}$. The involution of $F^{(i)}$ determines a unique $\iota(b_2^{(i)})$ and thus a unique $v^{(i)}$. 

This finishes the proof of Theorem \ref{theorem-full}.
\qed

\section{A model calculation of quadrilaterals}
\label{section-model}

We make a model calculation which will be used in Section \ref{section-markov} and \ref{section-example}. 
Consider the trivial fibration $\hat{p}\colon\mathbb{C}\times T^*S^{n-1}\to\mathbb{C}$ and Lagrangian submanifolds $a_i=\{y=i\}\times S^{n-1}$, $i=1,2$, $b_j=\{x=j\}\times S^{n-1}$, $j=1,2$, where $S^{n-1}$ is the zero section of $T^*S^{n-1}$. 
We further modify $a_i,b_j$ to $a'_i,b'_j$ by a Hamiltonian perturbation in the fiber direction so that they intersect transversely. 
Specifically, we choose the restriction of Euclidean metric on $S^{n-1}$ and identify $T^*S^{n-1}$ with $TS^{n-1}$. 
Choose Morse functions $f_1,f_2,f_3,f_4$ on $S^{n-1}$ each with 2 critical points and all of the critical points are disjoint (as the right of Figure \ref{quadrilateral}).
We can then rescale these Morse functions so that the difference of each pair is still Morse with 2 critical points:

\begin{figure}[ht]
    \centering
    \includegraphics[width=10cm]{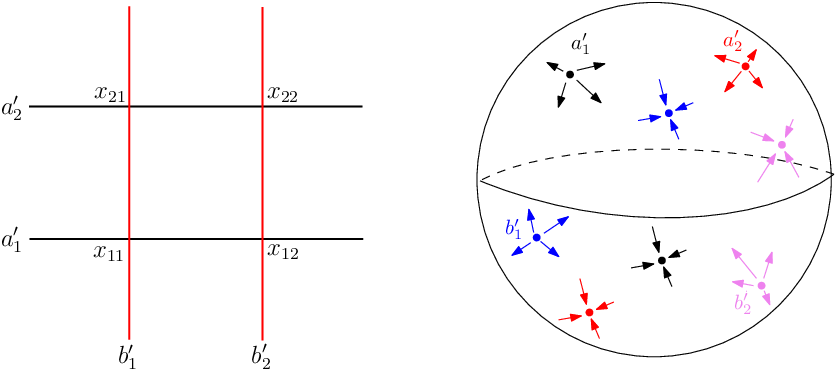}
    \caption{The base $\mathbb{C}$ on the left and the fiber $T^*S^{n-1}$ on the right.}
    \label{quadrilateral}
\end{figure}

\begin{lemma}
    For small enough $\epsilon>0$, the difference of each pair of functions in $\{\epsilon^3f_1,\epsilon^2f_2,\epsilon f_3,f_4\}$ is Morse with 2 critical points.
\end{lemma}

\noindent 
\textit{Proof}. For small enough $\epsilon>0$, $f_4-\epsilon f_3$ is a small perturbation of $f_4$. Since Morse condition is $C^\infty$-stable, $f_4-\epsilon f_3$ is still Morse with 2 critical points. By a simple induction the proof is finished.
\qed

~\\
Denote $f_{a_1'}=\epsilon^3 f_1,\,f_{a_2'}=\epsilon^2 f_2,\,f_{b_1'}=\epsilon f_3,\,f_{b_2'}=f_4$, the gradients of which correspond to the fiber projection of $a_1',a_2',b_1',b_2'$. Let $\check{x}_{ij},\hat{x}_{ij}$ over $x_{ij}$ be the top and bottom critical points of $f_{b_j'}-f_{a_i'}$, $i,j\in\{1,2\}$.

Now we compute the differentials of $\widehat{CF}(\boldsymbol{b'},\boldsymbol{a'})$, which is generated by 8 elements $\{x_{12}^\dag,x_{21}^\dag\}$ and $\{x_{11}^\dag,x_{22}^\dag\}$, where $\dag$ denotes a check or hat.

\begin{lemma}
    The differential of $\widehat{CF}(\boldsymbol{b'},\boldsymbol{a'})$ is given by $\hbar$ times the arrows in Figure \ref{markov-diff}. Moreover, a relative grading by Maslov index is denoted in Figure \ref{markov-diff}.
    \label{lemma-diff}
\end{lemma}

\begin{figure}[ht]
    \centering
    \includegraphics[height=3.5cm]{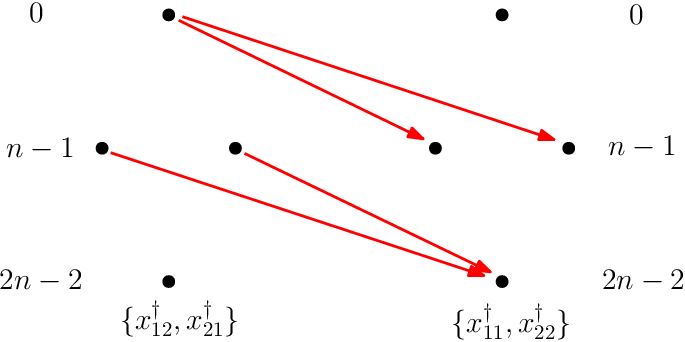}
    \caption{The differentials of $\widehat{CF}(\boldsymbol{b},\boldsymbol{a})$. The generators in the top row have 2 checks, those in the middle row have 1 check and those in the bottom row have no check.}
    \label{markov-diff}
\end{figure}
\begin{proof} 
    Let $u\colon\Dot{F}\to\mathbb{R}\times\left[0,1\right]\times(\mathbb{C}\times T^*S^{n-1})$ be a pseudoholomorphic disk with positive ends $\{{x}^{\dag}_{12},{x}^{\dag}_{21}\}$ and negative ends $\{{x}^{\dag}_{11},{x}^{\dag}_{22}\}$. Suppose the complex structure is split, then its projection to $\mathbb{C}$ is a degree 1 map over $[1,2]\times[1,2]$, which fixes the cross ratio of the 4 punctures on $\partial\Dot{F}$. 
    
    Then we consider the projection of $u$ to the fiber direction $T^*S^{n-1}$, denoted by $w\colon\Dot{F}\to T^*S^{n-1}$. 
    By the construction above $a'_1,a'_2,b'_1,b'_2$ are graphical near $S^{n-1}\subset T^*S^{n-1}$, with respect to Morse functions $f_{a'_i},f_{b'_j}$, $i,j={1,2}$. 
    The domain of the Morse moduli space is shown in Figure \ref{markov-morse}, where inner edges are ignored and arrows denote the direction of $-\nabla(f_\mathrm{right}-f_\mathrm{left})$. 
    
    \begin{figure}[ht]
        \centering
        \includegraphics[height=3.5cm]{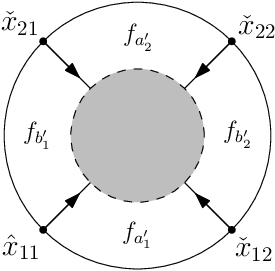}
        \caption{}
        \label{markov-morse}
    \end{figure}

    
    Viewing all boundary vertices as sources of gradient flow, observe that $\check{x}_{ij}$ is of Morse index $n-1$ if $i\neq j$ and 0 if $i=j$; $\hat{x}_{ij}$ is of Morse index $0$ if $i\neq j$ and $n-1$ if $i=j$. By Lemma \ref{lemma-morse}, we can further perturb $f_{a'_i},f_{b'_j}$, $i,j={1,2}$ such that all Morse gradient trees $w$ we consider are transversely cut out, and
    \begin{equation}
        \mathrm{ind}(w)=\left(\#\mathrm{checks\,\,in\,\,}\{x^{\dag}_{12},x^{\dag}_{21}\}+\#\mathrm{hats\,\,in\,\,}\{x^{\dag}_{11},x^{\dag}_{22}\}-3\right)(n-1)+1.
    \label{equation-index}
    \end{equation}
    Therefore, the Morse moduli space with respect to the arrows in Figure \ref{markov-diff} is of $\mathrm{ind}(w)=1$ and the case of gradient tree on $S^{n-1}$ from $\{\check{x}_{12},\check{x}_{21}\}$ to $\{\hat{x}_{11},\check{x}_{22}\}$ is shown in Figure \ref{markov-moduli}.
    
    \begin{figure}[ht]
        \centering
        \includegraphics[height=4.5cm]{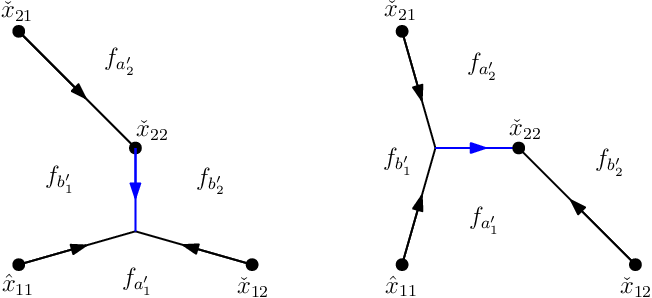}
        \caption{Two possible gradient trees on $S^{n-1}$. The moduli space is parametrized by the length of the inner (blue) edge.}
        \label{markov-moduli}
    \end{figure}
    
    By Theorem \ref{Fukaya-Oh}, the moduli space of pseudoholomorphic curves is diffeomorphic to the Morse moduli space, so we can think of gradient trees instead of pseudoholomorphic disks. Taking the base direction into consideration, one checks that 
    \begin{equation}
        \mathrm{ind}(u)=\mathrm{ind}(w).
        \label{equation-index1}
    \end{equation}
    For example, we still consider the case of Figure \ref{markov-moduli}: The two gradient trees are parametrized by the length of their inner edges. As the inner length tends to zero, the left and right gradient trees tend to the same one. As the inner edge of the left one tends to the bottom generator of $f_{b'_2}-f_{b'_1}$, its length tends to infinity and $q(\check{x}_{21})$ approaches $q(\check{x}_{22})$. Similarly, as the inner edge of the right one tends to the top generator of $f_{a'_2}-f_{a'_1}$, its length tends to infinity and $q(\check{x}_{12})$ approaches $q(\hat{x}_{22})$. 
    Since the cross ratio on the domain is fixed by the base direction, the result is that the algebraic count of $w$ is one. This verifies $\mathrm{ind}(u)=1$ and the arrows in Figure \ref{markov-diff}.
    
    If we set $\{\check{x}_{12},\check{x}_{21}\}$ to be of grading 0, we can verify the relative grading of generators in Figure \ref{markov-diff} by Lemma \ref{lemma-index}, (\ref{equation-index}), (\ref{equation-index1}) and the convention that $|\hbar|=2-n$, where the difference of grading is given by the Maslov index.
\end{proof}

\section{Invariance under Markov stabilization}
\label{section-markov}

A Markov stabilization is shown in Figure \ref{markov}: $\sigma$ is a $\kappa$-strand braid which intersects $D$ along $\boldsymbol{z}=\{z_1,\dots,z_k\}$. On the base $D$, $\sigma$ is viewed as an element of $\mathrm{Diff^+}(D,\partial D,z)$, which restricts to identity near $\gamma_0$. Without loss of generality, we construct a positive Markov stabilization between $\gamma_0$ and $\gamma_1$: Let $c$ be an arc from $z_0$ to $z_1$ which is disjoint from other $\gamma_j$, perform a positive half twist along $c$, then we get a $(\kappa+1)$-strand braid given by $\sigma\circ\sigma_c$.

Now we consider the fiber and Lagrangians. Let $p'\colon W'\to D$ be the standard Lefschetz fibration with regular fiber $T^*S^{n-1}$ and critical values $\boldsymbol{z'}=\{z_0,\dots,z_\kappa\}$ and $p\colon W\to D-N(\gamma_0)$ be its restriction to $D-N(\gamma_0)$. 
Let $a_j$ denote the Lagrangian thimble over $\gamma_j$.
Let $h_\sigma$ be an element of $\mathrm{Symp}(W,\partial W)$ which descends to $\sigma$ and $h'_\sigma\in \mathrm{Symp}(W',\partial W')$ be its extension to $W'$ by identity. 
Finally, let $\tau_c\in \mathrm{Symp}(W',\partial W')$ be the Dehn twist along the Lagrangian sphere over $c$. 

\begin{figure}[ht]
    \centering
    \includegraphics[width=12cm]{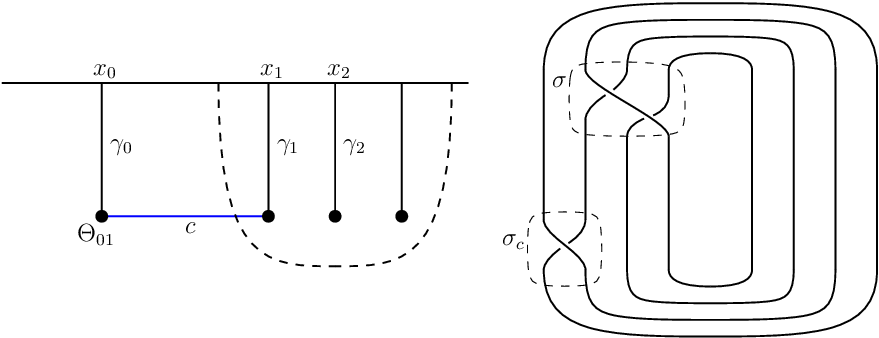}
    \caption{Markov stabilization along $c$.}
    \label{markov}
\end{figure}

\begin{figure}[ht]
    \centering
    \includegraphics[height=3cm]{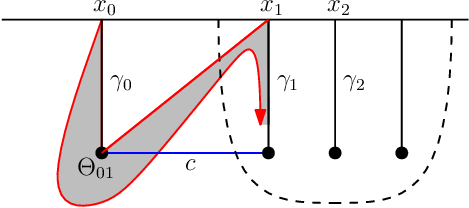}
    \caption{The red half-arcs are $\sigma\circ\sigma_c(\gamma_0)$ and $\sigma\circ\sigma_c(\gamma_1)$. The shaded region denotes the curve we are gluing.}
    \label{markov2}
\end{figure}

The proof of invariance under Markov stabilization is the same as Theorem 9.4.2 of \cite{colin2020applications}, and we briefly restate its proof here:

\begin{theorem}
    $\widehat{CF}(W,h_\sigma(\boldsymbol{a}),\boldsymbol{a})$ and $\widehat{CF}(W',h'_\sigma\circ\tau_c(\boldsymbol{a}'),\boldsymbol{a}')$ are isomorphic cochain complexes for specific choices of almost complex structure and $h_\sigma(\boldsymbol{a})$ and $h'_\sigma\circ\tau_c(\boldsymbol{a}')$ after a Hamiltonian isotopy.
\end{theorem}
\begin{proof}
    We directly construct a homomorphism $\Phi_s:\widehat{CF}(W,h_\sigma(\boldsymbol{a}),\boldsymbol{a})\to\widehat{CF}(W',h'_\sigma\circ\tau_c(\boldsymbol{a}'),\boldsymbol{a}')$ and show it is a cochain isomorphism. The notations are as in Figure \ref{markov2}.
    
    Consider the $\kappa$-tuples in $\widehat{CF}(W,h_\sigma(\boldsymbol{a}),\boldsymbol{a})$: It may or may not contain $\{x_1\}$. Similarly, the $\kappa$-tuples in $\widehat{CF}(W',h'_\sigma\circ\tau_c(\boldsymbol{a}'),\boldsymbol{a}')$ may contain either $\{x_0,x_1\}$ or $\{\Theta_{01}\}$. In fact there is a linear isomorphism:
    \begin{gather*}
        \Phi_s\colon\widehat{CF}(W,h_\sigma(\boldsymbol{a}),\boldsymbol{a})\to\widehat{CF}(W',h'_\sigma\circ\tau_c(\boldsymbol{a}'),\boldsymbol{a}'),\\
        \{x_1\}\cup\boldsymbol{y}'\mapsto\{x_0,x_1\}\cup\boldsymbol{y}',\,\,\boldsymbol{y}\mapsto\{\Theta_{01}\}\cup\boldsymbol{y},
    \end{gather*}
    
    For convenience of gluing below, we put $h_\sigma(a_1)$ and $h'_\sigma\circ\tau_c(a_0)$ in a position so that they go over the same arc near $\gamma_1$.
    Let $\gamma_1=\{\mathrm{Re}\,z_1\}\times[-1,0]\subset D$ and $\epsilon>0$ be small. The projections $\sigma(\gamma_1)$ and $\sigma\circ\sigma_c(\gamma_0)$ are written as $\zeta_1\cup\zeta_2\cup\zeta_3$ and $\zeta'_1\cup\zeta_2\cup\zeta_3$. The main requirement is that $\zeta_2$ be a common neck, for example, set $\zeta_2=\{\mathrm{Re}\,z_1-\epsilon\}\times[-2/3,-1/3]$. 
    Refer to Figure \ref{markov2} for details.
    
    We compare the differential of $\widehat{CF}(W',h'_\sigma\circ\tau_c(\boldsymbol{a}'),\boldsymbol{a}')$ and $\widehat{CF}(W,h_\sigma(\boldsymbol{a}),\boldsymbol{a})$. If $u'$ goes from $\{x_0,x_1\}\cup{y_1'}$ to $\{x_0,x_1\}\cup{y_2'}$, then $u'$ is in bijection with $u$ that goes from $\{x_1\}\cup{y_1'}$ to $\{x_1\}\cup{y_2'}$ since the strip from $x_0$ to $x_0$ is trivial. 
    Similarly, $u'$ that goes from $\{\Theta_{01}\}\cup{y_1}$ to $\{\Theta_{01}\}\cup{y_2}$ is in bijection with $u$ that goes from ${y_1}$ to ${y_2}$. There are no curves from $\{x_0,x_1\}\cup{y_1'}$ to $\{\Theta_{01}\}\cup{y_2}$ and no curves from $\{x_1\}\cup{y_1'}$ to ${y_2}$. The nontrivial case is that, if $u'$ goes from $\{\Theta_{01}\}\cup{y_1}$ to $\{x_0,x_1\}\cup{y_2'}$, then it is in bijection with $u$ that goes from ${y_1}$ to $\{x_0\}\cup{y_2'}$, where $u'$ comes from $u$ by replacing the end containing $x_1$ by the shaded region in Figure \ref{markov2}. 
    The bijection comes from Lemma \ref{lemma-diff} which says that the curve over the shaded region has algebraic count 1. The gluing details are omitted.
\end{proof}

\section{Examples}
\label{section-example}
In this section we consider some simple links and compute their cohomology groups $Kh^\sharp(\widehat{\sigma})$ in the sense of Theorem \ref{main-thm}, with coefficient ring $\mathbb{F}[\mathcal{A}]\llbracket\hbar,\hbar^{-1}]$ when $n=2$ and $\mathbb{F}\llbracket\hbar,\hbar^{-1}]$ when $n>3$, where $\hbar$ is of degree $2-n$. We assume that $\mathbb{F}$ is of characteristic 2 for simplicity.

\subsection{Unknots}
Figure \ref{unknot} shows the 2-strand braid representation of an unknot. In the Morse-Bott family of Reeb chords, $x_1,x_2$ are viewed as longer Reeb chords (top generators) and $z_1,z_2$ are viewed as shorter Reeb chords (bottom generators). The only possible differential in $\widehat{CF}(\widetilde{W},\widetilde{h}_{\sigma_{\mathrm{unknot}}} (\widetilde{\boldsymbol{a}}),\widetilde{\boldsymbol{a}})$ counts quadrilaterals $u$ with $\{x_1,x_2\}$ at positive ends and $\{z_1,z_2\}$ at negative ends. The projection of $u$ to $D$ has degree 1 over the region bounded by the loop $x_1\to z_2\to x_2\to z_1\to x_1$ and degree 0 over its complement. 

\begin{figure}[ht]
    \centering
    \includegraphics[height=3.6cm]{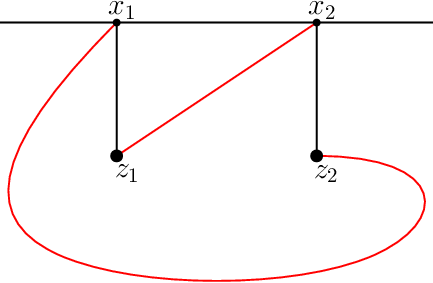}
    \caption{The braid representation of an unknot on $D$.}
    \label{unknot}
\end{figure}

\begin{proposition}
    $Kh^\sharp(\widehat{\sigma}_{\mathrm{unknot}})$ is freely generated by $\{x_1,x_2\}$ and $\{z_1,z_2\}$, where the difference of grading between these two generators is $2$ (mod $n-2$):
    
    \begin{center}
    \begin{tabular}{ |c|c| } 
        \hline
        generators & grading\\
        \hline
        $\{x_1,x_2\}$ & $0$ \\ 
        $\{z_1,z_2\}$ & $2$\\
        \hline
    \end{tabular}
    \end{center}
\end{proposition}
\begin{proof}
    Similar to the proof of Lemma \ref{lemma-diff}, we check that the pseudoholomorphic disk $u$ with $\{x_1,x_2\}$ at positive ends and $\{z_1,z_2\}$ at negative ends is of Fredholm index $n$ and Maslov index $2n-2$. Therefore, $\{x_1,x_2\}$ and $\{z_1,z_2\}$ are both cocycles and hence generators of $Kh^\sharp(\widehat{\sigma}_{\mathrm{unknot}})$.
\end{proof}

\subsection{Hopf links}
We then consider the 2-strand braid representation of a left-handed Hopf link as the left side of Figure \ref{hopf-link}.
`$A$', `$B$' and `$C$' denote the corresponding regions on the base $D$. For example, we use `$A+B$' to represent the union of region $`A\textrm'$ and $`B\textrm'$.

\begin{figure}[ht]
    \centering
    \includegraphics[height=4cm]{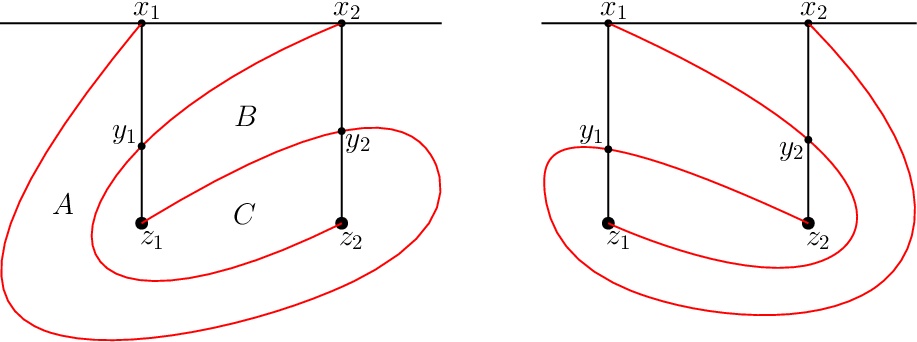}
    \caption{The braid representation of the left-handed Hopf link (left) and the right-handed Hopf link (right).}
    \label{hopf-link}
\end{figure}

\begin{lemma}
    \label{lemma-hopf}
    $\widehat{CF}(\widetilde{W},\widetilde{h}_{\sigma_\mathrm{left\,Hopf}}(\widetilde{\boldsymbol{a}}),\widetilde{\boldsymbol{a}})$ contains the following (mod 2) differential relations:
    \begin{align}
        d&\{x_1,z_2\}=\hbar\{\hat{y}_1,\check{y}_2\}+\hbar\{\check{y}_1,\hat{y}_2\},\label{diff-relation}\\
        d&\{x_2,z_1\}=\hbar\{\hat{y}_1,\check{y}_2\}+\hbar\{\check{y}_1,\hat{y}_2\},\label{diff-relation'}\\
        d&\{\check{y}_1,\hat{y}_2\}=\hbar\{z_1,z_2\},\label{diff-relation1}\\
        d&\{\hat{y}_1,\check{y}_2\}=\hbar\{z_1,z_2\}.\label{diff-relation2}
    \end{align}
\end{lemma}
\begin{proof}
    Note that $x_1,x_2,z_1,z_2$ should be viewed as top generators at positive ends or as bottom generators at negative ends. Specifically they give no constraints on the Morse-Bott family. 
    
    All nontrivial index 1 curves with count 1 (mod 2) are listed as follows, where we label the regions with positive weights after projection to the base $D$:
    \begin{center}
        \begin{tabular}{ |c|c| } 
         \hline
         regions & ends of generators\\
         \hline
         $A$ & $\{x_1,z_2\}\to\{\check{y}_1,\hat{y}_2\}/\{\hat{y}_1,\check{y}_2\}$\\ 
         $B$ & $\{x_2,z_1\}\to\{\check{y}_1,\hat{y}_2\}/\{\hat{y}_1,\check{y}_2\}$\\ 
         $C$ & $\{\check{y}_1,\hat{y}_1\}/\{\hat{y}_1,\check{y}_1\}\to\{z_1,z_2\}$\\
         \hline
        \end{tabular}
    \end{center}
    
    To see this, after projection to the base, the domain with positive weights is bounded by one of the following loops:
    \begin{enumerate}
        \item $y_1\to x_1\to y_2\to z_2\to y_1$. This corresponds to region `$A$', where $u$ is a quadrilateral with $\{x_1,z_2\}$ at positive ends and $\{\check{y}_1,\hat{y}_2\}$ or $\{\hat{y}_1,\check{y}_2\}$ at negative ends. By Lemma \ref{lemma-diff}, $\mathrm{ind}(u)=1$ and we get (\ref{diff-relation}).
        
        \item $y_1\to z_1\to y_2\to x_2\to y_1$. This corresponds to region `$B$', where $u$ is a quadrilateral with $\{x_2,z_1\}$ at positive ends and $\{\check{y}_1,\hat{y}_2\}$ or $\{\hat{y}_1,\check{y}_2\}$ at negative ends. This is similar to (1). Therefore, $\mathrm{ind}(u)=1$ and we get (\ref{diff-relation'}).
        
        \item $y_1\to z_2\to y_2\to z_1\to y_1$. This corresponds to region `$C$', where $u$ is a quadrilateral with $\{z_1,z_2\}$ at negative ends. 
        Since $z_1,z_2$ are bottom generators, we need one check and one hat at positive ends due to Lemma \ref{lemma-diff} so that there is a nontrivial count of $u$. Therefore we get (\ref{diff-relation1}) and (\ref{diff-relation2}).
        
        \item $y_1\to x_1\to y_2\to x_2\to y_1$. This corresponds to region `$A+B+C$', where $u$ is a quadrilateral with $\{x_1,x_2\}$ at positive ends and $\{{y}_1,{y}_2\}$ at negative ends. $x_1,x_2$ are viewed as top generators and there are two critical points inside the domain. We check that $\mathrm{ind}(u)=n$ for $\{\check{y}_1,\check{y}_2\}$ at negative ends; $\mathrm{ind}(u)=2n-1$ for $\{\hat{y}_1,\check{y}_2\}$ and $\{\check{y}_1,\hat{y}_2\}$ at negative ends; $\mathrm{ind}(u)=3n-2$ for $\{\hat{y}_1,\hat{y}_2\}$ at negative ends. Therefore there is no such $u$ with index 1.
    \end{enumerate}

\end{proof}

\begin{corollary}
    If we set $|\{x_1,x_2\}|=0$, then $Kh^\sharp(\widehat{\sigma}_{\mathrm{left\,Hopf}})$ is freely generated by the following generators with the corresponding relative grading (mod $n-2$):
    \begin{center}
    \begin{tabular}{ |c|c| } 
     \hline
     generators & grading\\
     \hline
     $\{x_1,x_2\}$ & $0$ \\ 
     $\{x_2,z_1\}+\{x_1,z_2\}$ & $2$\\
     $\{\check{y}_1,\check{y}_2\}$ & $2$ \\ 
     $\{\hat{y}_1,\hat{y}_2\}$ & $4$ \\ 
     \hline
    \end{tabular}
    \end{center}
\end{corollary}

~\\
The computation for the right-handed Hopf link is similar. As shown in Figure \ref{hopf-link}, the braid representation of the right-handed Hopf link is the mirror of the left-handed one. Note that there is a bijection of curves between these two Hopf links: Each curve in the left-handed moduli space corresponds to a curve in the right-handed one, where the positive and negative ends are exchanged, as well as the checks and hats. As a consequence, the grading by Maslov indices is also reversed. Specifically, the right-handed Hopf link satisfies:

\begin{lemma}
    $\widehat{CF}(\widetilde{W},\widetilde{h}_{\sigma_\mathrm{right\,Hopf}}(\widetilde{\boldsymbol{a}}),\widetilde{\boldsymbol{a}})$ contains the following (mod 2) differential relations:
    \begin{align*}
        d&\{\check{y}_1,\hat{y}_2\}=\hbar\{x_1,z_2\}+\hbar\{x_2,z_1\},\\
        d&\{\hat{y}_1,\check{y}_2\}=\hbar\{x_1,z_2\}+\hbar\{x_2,z_1\},\\
        d&\{z_1,z_2\}=\hbar\{\check{y}_1,\hat{y}_2\}+\hbar\{\hat{y}_1,\check{y}_2\}.
    \end{align*}
\end{lemma}

\begin{corollary}
    If we set $|\{x_1,x_2\}|=0$, then $Kh^\sharp(\widehat{\sigma}_{\mathrm{right\,Hopf}})$ is freely generated by the following generators with the corresponding relative grading (mod $n-2$):
    \begin{center}
    \begin{tabular}{ |c|c| } 
     \hline
     generators & grading\\
     \hline
     $\{x_1,x_2\}$ & $0$ \\ 
     $\{x_1,z_2\}$ & $-2$\\
     $\{\hat{y}_1,\hat{y}_2\}$ & $-2$ \\ 
     $\{\check{y}_1,\check{y}_2\}$ & $-4$ \\ 
     \hline
    \end{tabular}
    \end{center}
\end{corollary}

\subsection{Trefoils}
The left side of Figure \ref{trefoil} shows the 2-strand braid representation of a left-handed trefoil, where $`A\textrm'$ to $`E\textrm'$ denote the corresponding regions on the base.

\begin{figure}[ht]
    \centering
    \includegraphics[height=5cm]{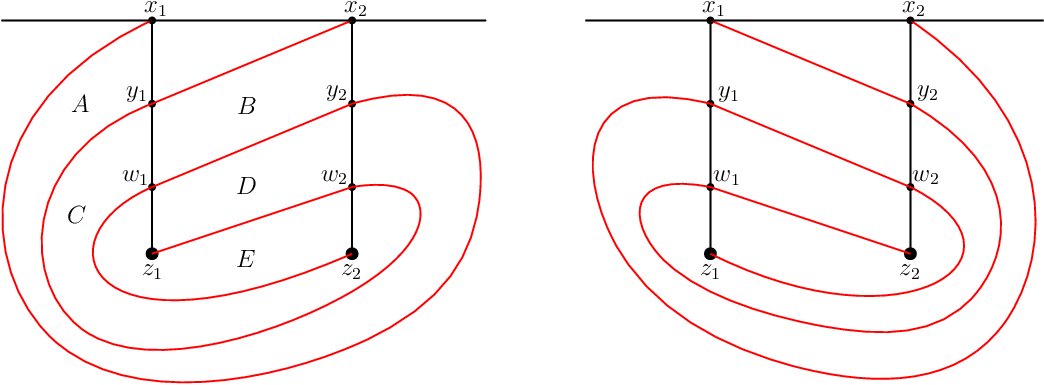}
    \caption{The braid representation of the left-handed trefoil (left) and the right-handed trefoil (right).}
    \label{trefoil}
\end{figure}

The curve counting for trefoils is more interesting than unknots and Hopf links. 
We first do a model calculation. 
Recall the notations in Step 3' of the proof of Theorem \ref{theorem-half}.
We consider curves with boundary on $T\cup L$: Let $\mathcal{M}'_2$ (2 stands for $n=2$) be the moduli space of holomorphic disks
\begin{equation*}
    u\colon\mathbb{R}\times[0,1]\to\mathbb{C}^2_{z_1',z_2'}
\end{equation*}
with standard complex structure $J_{std}$ satisfying
\begin{enumerate}
    \item $u(\mathbb{R}\times\{0\})\subset T$ and $u(\mathbb{R}\times\{1\})\subset L$;
    \item $p'\circ u$ has degree 1 over $\{|z|<1\}-\{-1\leq\mathrm{Re}\,z\leq0\}\cap\{\mathrm{Im}\,z=0\}$ and degree 0 otherwise;
    \item $u(-\infty,0)=w_1=(w_{11}',w_{12}')=(-1,1)$.
\end{enumerate}
It is easy to see that $\mathcal{M}'_2$ is homeomorphic to a line segment where $\partial\mathcal{M}'_2$ consists of two curves $z\mapsto(z,1)$ and $z\mapsto(-1,-z)$. 
Figure \ref{fig-model-trefoil} gives a schematic description of $\mathcal{M}'_2$ in the coordinates $(z_1',z_2')$.
For a given $\theta\in(-\pi,\pi)$, define the evaluation map
\begin{equation*}
    ev'_\theta\colon\mathcal{M}'_2\to S^1_{|z_1'|=1}
\end{equation*}
as the $z'_1$ projection of the intersection between $u$ and $p'^{-1}(e^{i\theta})$, which is shown by violet dots in the $z'_1$ column of Figure \ref{fig-model-trefoil}.

\begin{figure}[ht]
    \centering
    \includegraphics[width=8cm]{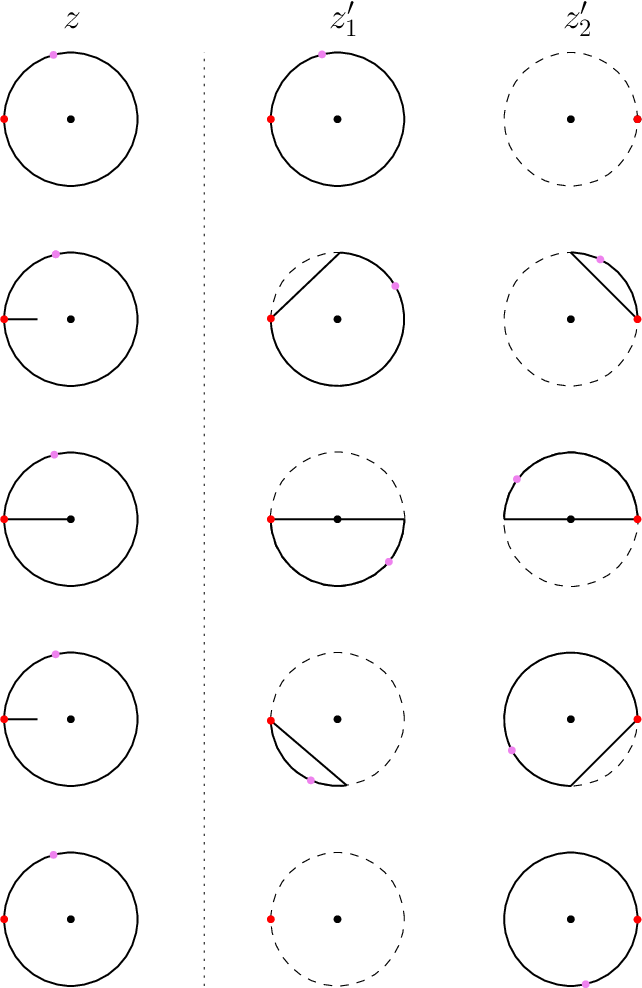}
    \caption{The left column denotes the base, where the red dot indicates $-1$, and the violet dot indicates $e^{i\theta}$.}
    \label{fig-model-trefoil}
\end{figure}

Note that $(-1,1)$ in $(z_1',z_2')$-coordinates equals $(0,i)$ in $(z_1,z_2)$-coordinates.
For general $n\geq2$, define $\mathcal{M}'_n$ similar to $\mathcal{M}'_2$, where $T,L$ in condition (1) are still the Lagrangian vanishing cycles over the unit circle and $\{-1\leq\mathrm{Re}\,z\leq0\}\cap\{\mathrm{Im}\,z=0\}$. Condition (3) is modified to
\begin{enumerate}
    \item[(3')] $u(-\infty,0)=w_1=(w_{11},w_{12},\dots,w_{1n})=(0,i,0,\dots,0)$.
\end{enumerate}

The moduli space $\mathcal{M}'_2$ in coordinate $(z_1,z_2)$ is viewed as the slice of $\mathcal{M}'_n$ that restricts to $0$ on $z_3,\dots,z_n$. 
Observe that for $\mathcal{M}'_n$, the coordinates $z_1,z_3,\dots,z_n$ are symmetric. 
Thus we can recover $\mathcal{M}'_n$ from $\mathcal{M}'_2$ by a symmetric rotation of $z_1$-coordinate. 
Each $u=(z_1,z_2)\in\mathcal{M}'_2$ corresponds to a $S^{n-2}$-family of curves in $\mathcal{M}'_n$:
\begin{equation*}
    u_{\lambda_1,\dots,\lambda_{n-1}}=(\lambda_1 z_1,z_2,\lambda_2 z_1,\dots,\lambda_{n-1} z_1),
\end{equation*}
where $\lambda_1^2+\dots+\lambda_{n-1}^2=1$ and $\lambda_i\in\mathbb{R}$ for $i=1,\dots,n-1$. 
    
The new evaluation map $ev_\theta$ is defined as the $n$-tuple of coordinates which projects to $e^{i\theta}\in\mathbb{C}_z$. 
Therefore, $\mathcal{M}'_n$ is homeomorphic to $D^{n-1}$. 
One can check that $ev_\theta\colon\mathcal{M}'_n\to S^{n-1}$ is a homeomorphism to its image.
As $\theta$ increases from $-\pi$ to $\pi$, the ratio of area of $ev_\theta(\mathcal{M}'_n)\subset S^{n-1}$ increases from 0 to 1.   
    

\begin{lemma}
    \label{lemma-trefoil}
    $\widehat{CF}(\widetilde{W},\widetilde{h}_{\sigma_\mathrm{left\,trefoil}}(\widetilde{\boldsymbol{a}}),\widetilde{\boldsymbol{a}})$ contains the following (mod 2) differential relations:
    \begin{align*}
        d&\{x_1,\check{w}_2\}=\hbar\{\check{y}_1,\hat{y}_2\}+\hbar\{\hat{y}_1,\check{y}_2\},\\
        d&\{x_1,\hat{w}_2\}=\hbar\{\hat{y}_1,\hat{y}_2\}+\hbar\{\check{y}_2,z_1\}+\hbar\{\check{y}_1,z_2\},\\
        d&\{x_2,\check{w}_1\}=\hbar\{\check{y}_1,\hat{y}_2\}+\hbar\{\hat{y}_1,\check{y}_2\},\\
        d&\{x_2,\hat{w}_1\}=\hbar\{\hat{y}_1,\hat{y}_2\}+\hbar\{\check{y}_1,z_2\}+\hbar\{\check{y}_2,z_1\},\\
        d&\{\check{y}_1,z_2\}=\hbar\{\check{w}_1,\hat{w}_2\}+\hbar\{\hat{w}_1,\check{w}_2\},\\
        d&\{\hat{y}_1,z_2\}=\hbar\{\hat{w}_1,\hat{w}_2\},\\
        d&\{\check{y}_2,z_1\}=\hbar\{\check{w}_1,\hat{w}_2\}+\hbar\{\hat{w}_1,\check{w}_2\},\\
        d&\{\hat{y}_2,z_1\}=\hbar\{\hat{w}_1,\hat{w}_2\},\\
        d&\{\check{w}_1,\hat{w}_2\}=\hbar\{z_1,z_2\},\\
        d&\{\hat{w}_1,\check{w}_2\}=\hbar\{z_1,z_2\},\\
        d&\{\check{y}_1,\hat{y}_2\}=\hbar\{\check{w}_1,\check{w}_2\},\\
        d&\{\hat{y}_1,\check{y}_2\}=\hbar\{\check{w}_1,\check{w}_2\}.
    \end{align*}
\end{lemma}

\begin{proof}   
    All nontrivial index 1 curves with count 1 (mod 2) are listed as follows:
    \begin{center}
    \begin{tabular}{ |c|c| } 
        \hline
        regions & ends of generators\\
        \hline
        $A$ & $\{x_1,\check{w}_2\}\to\{\check{y}_1,\hat{y}_2\}/\{\hat{y}_1,\check{y}_2\}$, $\{x_1,\hat{w}_2\}\to\{\hat{y}_1,\hat{y}_2\}$\\ 
        $B$ & $\{x_2,\check{w}_1\}\to\{\check{y}_1,\hat{y}_2\}/\{\hat{y}_1,\check{y}_2\}$, $\{x_2,\hat{w}_1\}\to\{\hat{y}_1,\hat{y}_2\}$\\ 
        $C$ & $\{\check{y}_1,z_2\}\to\{\check{w}_1,\hat{w}_2\}/\{\hat{w}_1,\check{w}_2\}$, $\{\hat{y}_1,z_2\}\to\{\hat{w}_1,\hat{w}_2\}$\\ 
        $D$ & $\{\check{y}_2,z_1\}\to\{\check{w}_1,\hat{w}_2\}/\{\hat{w}_1,\check{w}_2\}$, $\{\hat{y}_2,z_1\}\to\{\hat{w}_1,\hat{w}_2\}$\\ 
        $E$ & $\{\check{w}_1,\hat{w}_2\}\to\{z_1,z_2\}$, $\{\hat{w}_1,\check{w}_2\}\to\{z_1,z_2\}$\\
        $A+C+E$ & $\{x_1,\hat{w}_2\}\to\{\check{y}_2,z_1\}$\\
        $B+C+E$ & $\{x_2,\hat{w}_1\}\to\{\check{y}_2,z_1\}$\\
        $A+D+E$ & $\{x_1,\hat{w}_2\}\to\{\check{y}_1,z_2\}$\\
        $B+D+E$ & $\{x_2,\hat{w}_1\}\to\{\check{y}_1,z_2\}$\\
        $C+D+E$ & $\{\check{y}_1,\hat{y}_2\}/\{\hat{y}_1,\check{y}_2\}\to\{\check{w}_1,\check{w}_2\}$\\
        \hline
    \end{tabular}
    \end{center}
    
    First consider the region `$B+D+E$'. 
    The generators $\hat{w}_1,\check{y}_1$ are viewed as point constraints and $x_2,z_2$ impose no constraints on the Morse-Bott family.
    There are two possible slits: going down from ${w}_1$ or going to the right from ${w}_1$.
    If the slit goes down, we can apply the model of Figure \ref{fig-model-trefoil}, where $\hat{w}_1$ corresponds to the red dot $(0,i,0,\dots,0)$ and $\check{y}_1$ corresponds to the violet dot over $e^{i\theta}$ for some $\theta\in(-\pi,\pi)$.
    If ${w}_1,{y}_1$ are close on the base, then $\theta$ is close to $\pi$. 
    Therefore, $ev_\theta(\mathcal{M}'_n)\subset S^{n-1}$ takes almost all of $S^{n-1}$ and hence $\#ev^{-1}_{\theta}(\check{y}_1)=1$ mod $2$. 
    Note that the disparity between $\theta$ and $\pi$ prevents $\hat{w}_1,\check{y}_1$ from being close on the fiber direction (with respect to the symplectic parallel transport away from $z_1,z_2$), which contributes to the complement $S^{n-1}\backslash ev_\theta(\mathcal{M}'_n)$.
    However, even if ${w}_1,{y}_1$ are not close on the base, the count remains the same: As the slit goes to the right from $w_1$, $\hat{w}_1,\check{y}_1$ are getting closer on the fiber direction, and this part takes care of the complement $S^{n-1}\backslash ev_\theta(\mathcal{M}'_n)$.
    The case of `$A+C+E$', `$B+C+E$' and `$A+D+E$' are similar.

    \begin{figure}[ht]
        \centering
        \includegraphics[height=4cm]{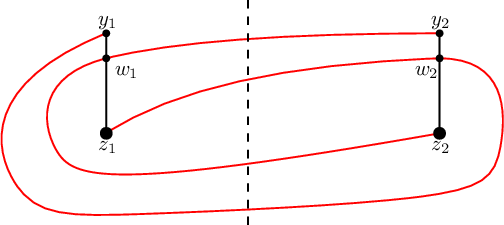}
        \caption{Horizontal stretch of region $C+D+E$.}
        \label{CDE}
    \end{figure}
    
    Next we consider the region `$C+D+E$'.
    We stretch the region `$C+D+E$' as depicted in Figure \ref{CDE}.
    The dashed line is viewed as a gluing condition of a point constraint of both sides.
    There are two different cases:
    \begin{enumerate}
        \item $\{\hat{y}_1,\check{y}_2\}\to\{\check{w}_1,\check{w}_2\}$. On the left side of Figure \ref{CDE}, $\hat{y}_1$ and $\check{w}_1$ are point constraints, which fixes a unique holomorphic disk by similar reasons as case `$B+D+E$'. Note that the slit going to the left from $w_1$ cannot be too long (across the dashed line): If the slit is long, then it cuts the left part into a trivial quadrilateral and a disk region with a branch point $z_1$ inside. However, the (mod 2) count of a holomorphic disk with a branch point inside is 0. 
        
        The left disk then fixes the point constraint of the dashed line. The right part of Figure \ref{CDE} now has two point constraints: $\check{w}_2$ and the dashed line. 
        We then apply the model of Figure \ref{fig-model-trefoil} again, where $\check{w}_2$ corresponds to the red dot and the dashed line corresponds to the violet dot. Similar to the discussion of `$B+D+E$', there exists a unique holomorphic disk.
        Note that the slit going to the left from $w_2$ cannot be too long for the same reason as above.
        To conclude, the (mod 2) count is 1.
        
        The case of $\{\check{y}_1,\hat{y}_2\}\to\{\check{w}_1,\check{w}_2\}$ is similar and the (mod 2) count is 1.\\

        \item $\{\hat{y}_1,\hat{y}_2\}\to\{\check{w}_1,\hat{w}_2\}$. Similar to (1), $\hat{y}_1$ and $\check{w}_1$ fix a unique holomorphic disk on the left side of Figure \ref{CDE} and gives the dashed line a point constraint.
        The dashed line and $\hat{y}_2$ are two point constraints on the right side. 
        If the slit from $w_2$ goes down, we apply the model of Figure \ref{fig-model-trefoil} where $\hat{y}_2$ corresponds to the red dot and the dashed line corresponds to the violet dot. 
        Note that $ev_\theta(\mathcal{M}'_n)$ sweeps approximately $1/4$ of $S^{n-1}$.
        One can check that the slit going to the left from $w_2$ also sweeps $1/4$ of $S^{n-1}$ (by symplectic parallel transport), and it cancels out with the contribution from $ev_\theta(\mathcal{M}'_n)$.
        As a result, the (mod 2) count is 0.

        The case of $\{\hat{y}_1,\hat{y}_2\}\to\{\hat{w}_1,\check{w}_2\}$ is similar and the (mod 2) count is 0.
    \end{enumerate}

    The remaining cases are similar to the discussion of Hopf links and we omit the details.

\end{proof}

\begin{corollary}
    If we set $|\{x_1,x_2\}|=0$, then $Kh^\sharp(\widehat{\sigma}_{\mathrm{left\,trefoil}})$ is freely generated by the following generators with the corresponding relative grading (mod $n-2$):
    \begin{center}
    \begin{tabular}{ |c|c| } 
     \hline
     generators & grading\\
     \hline
     $\{x_1,x_2\}$ & $0$ \\ 
     $\{x_1,\check{w}_2\}+\{x_2,\check{w}_1\}$ & $2$\\
     $\{\check{y}_1,\check{y}_2\}$ & $2$ \\   
     $\{x_1,\hat{w}_2\}+\{x_2,\hat{w}_1\}$ & $3$\\
     $\{\check{y}_1,z_2\}+\{\check{y}_2,z_1\}$ & $4$\\
     $\{\hat{y}_1,z_2\}+\{\hat{y}_2,z_1\}$ & $5$\\
     \hline
    \end{tabular}
    \end{center}
\end{corollary}

~\\
The computation for the right-handed trefoil is similar and the proof is omitted:

\begin{lemma}
    $\widehat{CF}(\widetilde{W},\widetilde{h}_{\sigma_\mathrm{right\,trefoil}}(\widetilde{\boldsymbol{a}}),\widetilde{\boldsymbol{a}})$ contains the following (mod 2) differential relations:
    \begin{align*}
        d&\{\check{y}_1,\hat{y}_2\}=\hbar\{x_1,\hat{w}_2\}+\hbar\{x_2,\hat{w}_1\},\\
        d&\{\hat{y}_1,\check{y}_2\}=\hbar\{x_1,\hat{w}_2\}+\hbar\{x_2,\hat{w}_1\},\\
        d&\{\check{y}_1,\check{y}_2\}=\hbar\{x_1,\check{w}_2\}+\hbar\{x_2,\check{w}_1\},\\
        d&\{\check{w}_1,\hat{w}_2\}=\hbar\{\hat{y}_1,z_2\}+\hbar\{\hat{y}_2,z_1\},\\
        d&\{\hat{w}_1,\check{w}_2\}=\hbar\{\hat{y}_1,z_2\}+\hbar\{\hat{y}_2,z_1\},\\
        d&\{\check{w}_1,\check{w}_2\}=\hbar\{\check{y}_1,z_2\}+\hbar\{\check{y}_2,z_1\},\\
        d&\{z_1,z_2\}=\hbar\{\check{w}_1,\hat{w}_2\}+\hbar\{\hat{w}_1,\check{w}_2\},\\
        d&\{\hat{y}_1,z_2\}=\hbar\{x_1,\check{w}_2\}+\hbar\{x_2,\check{w}_1\},\\
        d&\{\hat{y}_2,z_1\}=\hbar\{x_1,\check{w}_2\}+\hbar\{x_2,\check{w}_1\},\\
        d&\{\hat{w}_1,\hat{w}_2\}=\hbar\{\check{y}_1,\hat{y}_2\}+\hbar\{\hat{y}_1,\check{y}_2\}.
    \end{align*}
\end{lemma}

\begin{corollary}
    If we set $|\{x_1,x_2\}|=0$, then $Kh^\sharp(\widehat{\sigma}_{\mathrm{right\,trefoil}})$ is freely generated by the following generators with the corresponding relative grading (mod $n-2$):
    \begin{center}
    \begin{tabular}{ |c|c| } 
     \hline
     generators & grading\\
     \hline
     $\{x_1,x_2\}$ & $0$ \\ 
     $\{x_1,\hat{w}_2\}$ & $-2$\\
     $\{\hat{y}_1,\hat{y}_2\}$ & $-2$ \\   
     $\{x_1,\check{w}_2\}$ & $-3$\\
     $\{\hat{y}_1,z_2\}+\{\hat{y}_2,z_1\}$ & $-4$\\
     $\{\check{y}_1,z_2\}$ & $-5$\\
     \hline
    \end{tabular}
    \end{center}
\end{corollary}

~\\

\printbibliography

@misc{colin2020applications,
    title={Applications of higher\hyp{}dimensional Heegaard Floer homology to contact topology},
    author={Vincent Colin and Ko Honda and Yin Tian},
    year={2020},
    eprint={2006.05701},
    archivePrefix={arXiv},
    primaryClass={math.SG}
}

@article {fukaya1997zero,
    AUTHOR = {Fukaya, Kenji and Oh, Yong-Geun},
     TITLE = {Zero-loop open strings in the cotangent bundle and {M}orse
              homotopy},
   JOURNAL = {Asian J. Math.},
  FJOURNAL = {Asian Journal of Mathematics},
    VOLUME = {1},
      YEAR = {1997},
    NUMBER = {1},
     PAGES = {96--180},
      ISSN = {1093-6106},
   MRCLASS = {58E05 (58D29 58F09)},
  MRNUMBER = {1480992},
MRREVIEWER = {Joa Weber},
       DOI = {10.4310/AJM.1997.v1.n1.a5},
       URL = {https://doi.org/10.4310/AJM.1997.v1.n1.a5},
}

@book {fukaya2010lagrangian,
    AUTHOR = {Fukaya, Kenji and Oh, Yong-Geun and Ohta, Hiroshi and Ono,
              Kaoru},
     TITLE = {Lagrangian intersection {F}loer theory: anomaly and
              obstruction. {P}art {II}},
    SERIES = {AMS/IP Studies in Advanced Mathematics},
    VOLUME = {46},
 PUBLISHER = {American Mathematical Society, Providence, RI},
      YEAR = {2009},
     PAGES = {i--xii and 397--805},
      ISBN = {978-0-8218-4837-1},
   MRCLASS = {53D40 (53D12)},
  MRNUMBER = {2548482},
MRREVIEWER = {Michael J. Usher},
       DOI = {10.1090/crmp/049/07},
       URL = {https://doi.org/10.1090/crmp/049/07},
}

@article{lipshitz2006cylindrical,
    AUTHOR = {Lipshitz, Robert},
     TITLE = {A cylindrical reformulation of {H}eegaard {F}loer homology},
   JOURNAL = {Geom. Topol.},
  FJOURNAL = {Geometry and Topology},
    VOLUME = {10},
      YEAR = {2006},
     PAGES = {955--1096},
      ISSN = {1465-3060},
   MRCLASS = {57R58 (57M27)},
  MRNUMBER = {2240908},
MRREVIEWER = {Stanislav Jabuka},
       DOI = {10.2140/gt.2006.10.955},
       URL = {https://doi.org/10.2140/gt.2006.10.955},
}

@misc{colin2012equivalence,
    title={The equivalence of Heegaard Floer homology and embedded contact homology via open book decompositions I},
    author={Vincent Colin and Paolo Ghiggini and Ko Honda},
    year={2012},
    eprint={1208.1074},
    archivePrefix={arXiv},
    primaryClass={math.GT}
}

@article {ozsvath2004holomorphic,
    AUTHOR = {Ozsv\'{a}th, Peter and Szab\'{o}, Zolt\'{a}n},
     TITLE = {Holomorphic disks and topological invariants for closed
              three-manifolds},
   JOURNAL = {Ann. of Math. (2)},
  FJOURNAL = {Annals of Mathematics. Second Series},
    VOLUME = {159},
      YEAR = {2004},
    NUMBER = {3},
     PAGES = {1027--1158},
      ISSN = {0003-486X},
   MRCLASS = {57M27 (32Q65 57R58)},
  MRNUMBER = {2113019},
MRREVIEWER = {Thomas E. Mark},
       DOI = {10.4007/annals.2004.159.1027},
       URL = {https://doi.org/10.4007/annals.2004.159.1027},
}

@article {manolescu2006nilpotent,
    AUTHOR = {Manolescu, Ciprian},
     TITLE = {Nilpotent slices, {H}ilbert schemes, and the {J}ones
              polynomial},
   JOURNAL = {Duke Math. J.},
  FJOURNAL = {Duke Mathematical Journal},
    VOLUME = {132},
      YEAR = {2006},
    NUMBER = {2},
     PAGES = {311--369},
      ISSN = {0012-7094},
   MRCLASS = {53D40 (14C05 57M27 57R17)},
  MRNUMBER = {2219260},
MRREVIEWER = {Jacob Andrew Rasmussen},
       DOI = {10.1215/S0012-7094-06-13224-6},
       URL = {https://doi.org/10.1215/S0012-7094-06-13224-6},
}

@misc{perutz2008hamiltonian,
  title={Hamiltonian handleslides for Heegaard Floer homology},
  author={Perutz, Timothy},
  eprint={0801.0564},
  archivePrefix={arXiv},
  year={2008},
  primaryClass={math.SG}
}

@misc{mak2020fukayaseidel,
      title={Fukaya-Seidel categories of Hilbert schemes and parabolic category $\mathcal{O}$}, 
      author={Cheuk Yu Mak and Ivan Smith},
      year={2019},
      eprint={1907.07624},
      archivePrefix={arXiv},
      primaryClass={math.SG}
}

@article {seidel2006link,
    AUTHOR = {Seidel, Paul and Smith, Ivan},
     TITLE = {A link invariant from the symplectic geometry of nilpotent
              slices},
   JOURNAL = {Duke Math. J.},
  FJOURNAL = {Duke Mathematical Journal},
    VOLUME = {134},
      YEAR = {2006},
    NUMBER = {3},
     PAGES = {453--514},
      ISSN = {0012-7094},
   MRCLASS = {53D40 (57M27)},
  MRNUMBER = {2254624},
MRREVIEWER = {Michael J. Usher},
       DOI = {10.1215/S0012-7094-06-13432-4},
       URL = {https://doi.org/10.1215/S0012-7094-06-13432-4},
}

@article {khovanov2000categorification,
    AUTHOR = {Khovanov, Mikhail},
     TITLE = {A categorification of the {J}ones polynomial},
   JOURNAL = {Duke Math. J.},
  FJOURNAL = {Duke Mathematical Journal},
    VOLUME = {101},
      YEAR = {2000},
    NUMBER = {3},
     PAGES = {359--426},
      ISSN = {0012-7094},
   MRCLASS = {57M27 (57R56)},
  MRNUMBER = {1740682},
       DOI = {10.1215/S0012-7094-00-10131-7},
       URL = {https://doi.org/10.1215/S0012-7094-00-10131-7},
}

@article{williams2008computations,
  title={Computations in Khovanov Homology},
  author={Williams, Brandon},
  journal={University of Illinois at Chicago, Chicago, Illinois},
  year={2008}
}

@article{ekholm2013knot,
  title={Knot contact homology},
  author={Ekholm, Tobias and Etnyre, John B and Ng, Lenhard and Sullivan, Michael G},
  journal={Geometry \& Topology},
  volume={17},
  number={2},
  pages={975--1112},
  year={2013},
  publisher={Mathematical Sciences Publishers}
}

@misc{iacovino2008simple,
  title={A simple proof of a theorem of Fukaya and Oh},
  author={Iacovino, Vito},
  eprint={0812.0129},
  year={2008},
  archivePrefix={arXiv},
  primaryClass={math.SG}
}

@article {manolescu2007link,
    AUTHOR = {Manolescu, Ciprian},
     TITLE = {Link homology theories from symplectic geometry},
   JOURNAL = {Adv. Math.},
  FJOURNAL = {Advances in Mathematics},
    VOLUME = {211},
      YEAR = {2007},
    NUMBER = {1},
     PAGES = {363--416},
      ISSN = {0001-8708},
   MRCLASS = {57R58 (53D40 57M27)},
  MRNUMBER = {2313538},
MRREVIEWER = {Jacob Andrew Rasmussen},
       DOI = {10.1016/j.aim.2006.09.007},
       URL = {https://doi.org/10.1016/j.aim.2006.09.007},
}

@article {floer1988morse,
    AUTHOR = {Floer, Andreas},
     TITLE = {Morse theory for {L}agrangian intersections},
   JOURNAL = {J. Differential Geom.},
  FJOURNAL = {Journal of Differential Geometry},
    VOLUME = {28},
      YEAR = {1988},
    NUMBER = {3},
     PAGES = {513--547},
      ISSN = {0022-040X},
   MRCLASS = {58F05 (35J65 58E05)},
  MRNUMBER = {965228},
MRREVIEWER = {Jean-Claude Sikorav},
       URL = {http://projecteuclid.org/euclid.jdg/1214442477},
}

\end{document}